\documentclass[11pt]{amsart}

\usepackage{epigamath-voisin}

\usepackage[english]{babel}

\numberwithin{equation}{section}

\usepackage[shortlabels]{enumitem}

\usepackage{amsfonts}
\usepackage{latexsym,bm}
\usepackage{amssymb}
\usepackage{amsmath}
\usepackage{amsthm}
\usepackage{mathrsfs}
\usepackage[all]{xy}
\usepackage{xcolor}
\usepackage{tikz}
\usepackage{mathtools}

\theoremstyle{plain}
\newtheorem{thm}{Theorem}[section]
\newtheorem{lem}[thm]{Lemma}
\newtheorem{prop}[thm]{Proposition}
\newtheorem{cor}[thm]{Corollary}
\newtheorem{lem-defi}[thm]{Lemma-Definition}
\newtheorem{prop-defi}[thm]{Proposition-Definition}

\theoremstyle{remark}
\newtheorem{rmk}[thm]{Remark}

\theoremstyle{definition}

\setlist[enumerate,1]{label={\rm(\roman*)}, ref={\rm\roman*}}

\newcommand{\mcC}{{\mathcal C}}
\newcommand{\mcD}{{\mathcal D}}
\newcommand{\mcE}{{\mathcal E}}

\newcommand{\mcH}{{\mathcal H}}

\newcommand{\mcM}{{\mathcal M}}

\newcommand{\mcO}{{\mathcal O}}

\newcommand{\mcS}{{\mathcal S}}
\newcommand{\mcT}{{\mathcal T}}
\newcommand{\mcU}{{\mathcal U}}
\newcommand{\mcV}{{\mathcal V}}
\newcommand{\mcW}{{\mathcal W}}
\newcommand{\mcX}{{\mathcal X}}
\newcommand{\mcY}{{\mathcal Y}}

\newcommand{\mbA}{{\mathbb A}}

\newcommand{\mbC}{{\mathbb C}}

\newcommand{\mbF}{{\mathbb F}}
\newcommand{\mbG}{{\mathbb G}}
\newcommand{\mbH}{{\mathbb H}}

\newcommand{\mbO}{{\mathbb O}}
\newcommand{\mbP}{{\mathbb P}}
\newcommand{\mbQ}{{\mathbb Q}}
\newcommand{\mbR}{{\mathbb R}}
\newcommand{\mbS}{{\mathbb S}}

\newcommand{\mbZ}{{\mathbb Z}}

\newcommand{\bH}{{\textbf H}}

\newcommand{\bP}{{\mathbb P}}

\DeclareMathOperator{\id}{id}
\DeclareMathOperator{\rk}{rk}
\DeclareMathOperator{\com}{com}
\DeclareMathOperator{\diag}{diag}
\DeclareMathOperator{\lin}{lin}
\DeclareMathOperator{\Ad}{Ad}
\DeclareMathOperator{\rank}{rank}

\DeclareMathOperator{\NE}{NE}

\newcommand{\fh}{{\mathfrak h}}
\newcommand{\fk}{{\mathfrak k}}
\newcommand{\fl}{{\mathfrak l}}

\newcommand{\aut}{{\mathfrak a}{\mathfrak u}{\mathfrak t}}

\newcommand{\BP}{{\mathbb P}}

\def\RatCurves{\mathop{\rm RatCurves}\nolimits}
\def\Aut{\mathop{\rm Aut}\nolimits}

\def\GL{\mathop{\rm GL}\nolimits}
\def\Gr{\mathop{\rm Gr}\nolimits}

\def\Sp{\mathop{\rm Sp}\nolimits}
\def\Lag{\mathop{\rm Lag}\nolimits}

\def\Pic{\mathop{\rm Pic}\nolimits}
\def\End{\mathop{\rm End}\nolimits}

\def\SL{\mathop{\rm SL}\nolimits}
\def\SO{\mathop{\rm SO}\nolimits}

\def\tr{\mathop{\rm tr}\nolimits}
\def\det{\mathop{\rm det}\nolimits}

\def\Id{\mathop{\rm Id}\nolimits}
\def\Inn{\mathop{\rm Inn}\nolimits}

\newcommand{\sC}{{\mathcal C}}

\newcommand{\SJ}{{\mathcal J}}
\newcommand{\sJ}{{\mathcal J}}
\newcommand{\sK}{{\mathcal K}}

\newcommand{\sU}{{\mathcal U}}

\newcommand{\n}{{\rm n}}

\def\lra{\longrightarrow}

\def\shortbar{%
\smash{\scalebox{0.4}[1.0]{$-$}}}
\makeatother

\makeatletter
\newcommand{\xdashrightarrow}[2][]{\ext@arrow 0359\rightarrowfill@@{#1}{#2}}
\def\rightarrowfill@@{\arrowfill@@\relax\shortbar\dashrightarrow}
\def\arrowfill@@#1#2#3#4{%
  $\m@th\thickmuskip0mu\medmuskip\thickmuskip\thinmuskip\thickmuskip
   \relax#4#1
   \xleaders\hbox{$#4#2$}\hfill
   #3$%
}
\makeatother

\newcommand{\longdashrightarrow}{\xdashrightarrow{\hphantom{0pt}}}

\YearArticle{2023} \EpigaArticleNr{4} \ReceivedOn{December 7, 2022}
\InFinalFormOn{March 11, 2023}
\AcceptedOn{April 3, 2023}

\title{Rigidity of projective symmetric manifolds of Picard number 1 associated to composition algebras}
\titlemark{Rigidity of symmetric varieties}

\author{Yifei Chen}
\address{AMSS and HLM, Chinese Academy of Sciences, 55 ZhongGuanCun East Road, Beijing, 100190, China and School of Mathematical Sciences, University of Chinese Academy of Sciences, Beijing, China }
\email{yifeichen@amss.ac.cn}

\author{Baohua Fu}
\address{AMSS, HLM and MCM, Chinese Academy of Sciences, 55 ZhongGuanCun East Road, Beijing, 100190, China and School of Mathematical Sciences, University of Chinese Academy of Sciences, Beijing, China }
\email{bhfu@math.ac.cn}

\author{Qifeng Li}
\address{School of Mathematics, Shandong University, Jinan, 250100, China}
\email{qifengli@sdu.edu.cn}

\authormark{Y.~Chen, B.~Fu, and Q.~Li}

\AbstractInEnglish{
To each complex composition algebra $\mbA$ is associated a projective
symmetric manifold $X(\mbA)$ of Picard number $1$, which is just a
smooth hyperplane section of one of the following varieties:
$$\Lag(3,6), \quad \Gr(3,6), \quad \mbS_6, \quad E_7/P_7.$$
In this paper, we prove that these varieties are rigid; namely, for any smooth family of projective manifolds over a connected base, if one fiber is isomorphic to $X(\mbA)$, then every fiber is isomorphic to  $X(\mbA)$.
}


\acknowledgement{Y.~Chen is partially supported by the NSFC grant No. 12271384.
B.~Fu is supported by the NSFC grant No. 12288201 and CAS Project for Young Scientists in Basic Research grant No. YSBR-033. Q.~Li is supported by the NSFC grant No. 12201348.}

\dedication{Pour Claire, avec admiration}

\begin{document}


\maketitle

\begin{prelims}

\DisplayAbstractInEnglish

\bigskip

\DisplayKeyWords

\medskip

\DisplayMSCclass

\end{prelims}


\newpage

\setcounter{tocdepth}{1}

\tableofcontents


\section{Introduction}

Throughout this paper, we work over the complex number field. A smooth projective variety $X$ is said to be {\em rigid} if for any smooth projective family over a connected base with one fiber isomorphic to $X$, all fibers are isomorphic to $X$.  It is a difficult and subtle problem to prove the rigidity.  Even for rational homogeneous varieties $G/P$ of Picard number $1$, the rigidity does not always hold. To wit, let $B_3/P_2$ be the variety of lines on a $5$-dimensional smooth hyperquadric $\mathbb{Q}^5$. An explicit family specializing $B_3/P_2$ to a smooth projective $G_2$-variety is constructed by Pasquier and Perrin in \cite{PP}.  In \cite{HL}, it is shown that this is the only smooth non-isomorphic specialization of $B_3/P_2$.  It turns out that $B_3/P_2$ is the only exception among all $G/P$ of Picard number $1$, as shown by the following.

\begin{thm}[\textit{cf.}~\cite{Hw97, HM98, HM02, HM}] \label{t.HMRigidity}
 A rational homogeneous variety of Picard number $1$ is rigid except in the case of $B_3/P_2$.
\end{thm}

The key ingredient for the proof is the VMRT theory developed by Hwang and Mok.  In the simplest case of a projective manifold $X$ covered by lines (which is the case for our paper), the VMRT $\sC_x \subset \mbP T_xX$ at a general point $x$ is just the Hilbert scheme of lines through $x$. This projective subvariety $\sC_x \subset \mbP T_xX$ encodes a lot of global geometry of $X$, and in some cases, we can even recognize $X$ from its VMRT at general points.

As $G/P$ is locally rigid, we only need to prove that $G/P$ is rigid under specialization; namely, for a smooth projective family $\mcX \to \Delta$ such that $\mcX_t \simeq G/P$ for all $t \neq 0$, we have $\mcX_0 \simeq G/P$. The proof essentially consists of two steps: the first is to show the VMRT of $\mcX_0$ is isomorphic to that of $G/P$, and the second is to use the recognization of $G/P$ from its VMRT.

In many cases, the invariance of the VMRT can be proved, while the recognization problem is in general much more difficult. In \cite{Par}, it is observed that for odd Lagrangian Grassmannians (which are not homogeneous), one can directly show that $H^1(\mcX_0, T_{\mcX_0})=0$ by using VMRT theory. Hence $\mcX_0$ is locally rigid and isomorphic to nearby fibers, which proves the rigidity for odd Lagrangian Grassmannians.

The goal of this paper is to prove the rigidity for projective symmetric varieties associated to composition algebras. Recall that there are exactly four complex composition algebras: $\mbA = \mbC$, $\mbC\oplus \mbC$, $\mbH_\mbC$, $\mbO_\mbC$. To  such an $\mbA$, we can associate algebraic groups $\SL_3(\mbA)$ and $\SO_3(\mbA)$ with an involution $\theta$ such that $\SL_3(\mbA)^\theta = \SO_3(\mbA)$. The quotient $\SL_3(\mbA)/\SO_3(\mbA)$ is a symmetric homogeneous space, which admits a unique smooth equivariant completion of Picard number $1$, denoted by $X(\mbA)$. It turns out that $X(\mbA)$ is a smooth hyperplane section of one of the following varieties (\textit{cf.} \cite{R10}):
$$
\Lag(3,6), \quad \Gr(3,6), \quad \mbS_6, \quad E_7/P_7,
$$
where $\Lag(3,6)$ is the Lagrangian Grassmannian associated to $\mbC^6$ and $\mbS_6$ is the $15$-dimensional spinor variety.  The main result of this paper is the following.

\begin{thm} \label{t.main}
For any complex composition algebra $\mbA$, the variety $X(\mbA)$ is rigid.
\end{thm}

We first remark that for $\mbA = \mbC$, $X(\mbA)$ is a Mukai variety, so its smooth deformation is again a Mukai variety, hence again a hyperplane section of $\Lag(3,6)$ by the classification of Mukai varieties. This shows that $X(\mbA)$ is rigid in this case.  We will assume $\mbA \neq \mbC$ in the following.

The rigidity problem of $X(\mbA)$ was studied by Kim and Park in \cite{KP}. When $\mbA = \mbH_\mbC$ or $\mbO_\mbC$, they prove the invariance of the VMRT and, moreover, observe that $\dim H^1(\mcX_0, T_{\mcX_0}) \leq 1$.  If $H^1(\mcX_0, T_{\mcX_0}) =0$, then $\mcX_0$ is locally rigid, and thus it is isomorphic to nearby fibers $X(\mbA)$. When $\dim H^1(\mcX_0, T_{\mcX_0}) = 1$, $\mcX_0$ is an equivariant compactification of the vector group $\mathbb{G}_a^n$ with $n=\dim X(\mbA)$. With the help of \cite{W}, this result can be easily extended to the case $\mbA = \mbC \oplus \mbC$.

To prove Theorem~\ref{t.main}, we will exclude the case of equivariant compactifications.  Let $\mcX \to \Delta$ be a specialization of $X(\mbA)$; \textit{i.e.} $\mcX_t \simeq X(\mbA)$ for all $t \neq 0$ and $\mcX_0$ is an equivariant compactification of $\mathbb{G}_a^n$. The Lie algebra of the automorphism group of the central fiber is given by $\mathfrak{aut}(\mcX_0) \simeq \mathbb{C}^n\rtimes (\mathfrak{so}_3(\mbA) \oplus \mathbb{C})$.  We will consider a family of tori ${\mathbf H}_t \subset \Aut^0(\mcX_t)$ induced from a maximal torus of $\SO_3(\mbA)$ and then take a connected component $\mcY$ of the torus-fixed locus $\mcX^{\mathbf{H}}$.  As the rank of $\mathfrak{sl}_3(\mbA)$ is $2$ more than that of $\mathfrak{so}_3(\mbA)$, there is an extra $2$-dimensional torus acting on $\mcY_t$ for $t \neq 0$.  It turns out that $\mcY \to \Delta$ is a family of smooth projective surfaces with general fiber $\mcY_t$ isomorphic to the blowup of $\mathbb{P}^2$ along three coordinate points. The central fiber $\mcY_0$ is an equivariant compactification of $\mathbb{G}_a^2$. By delicate computations, we will show that $\mcY_0$ is isomorphic to the blowup of $\mathbb{P}^2$ along three colinear points.  On the other hand, the involution $\theta$ on $\SL_3(\mbA)$ induces an involution $\Theta$ on the family $\mcY/\Delta$, which preserves the boundaries of $\mcY_t$ for all $t$.  It turns out that the involution $\Theta_0\colon \mcY_0 \to \mcY_0$ sends extremal rays of the Mori cone $\overline{\NE}(\mathcal{Y}_0)$ to non-extremal rays, which gives a contradiction.

\subsection*{Acknowledgements}
We are very grateful to the referee for their careful reading and helpful suggestions.

\section{Projective symmetric manifolds of Picard number 1 associated to composition algebras}

\subsection{Composition algebras and associated Lie groups}
Let $\mbA_\mbR$ be one of the four real normed division algebras, $\mbR$, $\mbC$, $\mbH$, $\mbO$, which admits an involution $x \mapsto \bar{x}$, called conjugation.  It is well known that the fixed points under this conjugation are exactly the base field $\mbR$. Note that $\mbH$, $\mbO$ are non-commutative and $\overline{ab} = \bar{b} \bar{a}$ for all $a, b \in \mbA_\mbR$.

Let $\mbA = \mbA_\mbR \otimes_\mbR \mbC$ be the complexification of $\mbA_\mbR$, which is one of the following: $\mbC$, $\mbC\oplus \mbC$, $\mbH_\mbC$, $\mbO_\mbC$. The algebra structure on $\mbA$ is given by $(a\otimes c, a' \otimes c') \mapsto aa' \otimes cc'$ for multiplication and $\overline{a \otimes c} = \bar{a} \otimes \bar{c}$ for conjugation.  Note that the conjugation fixes exactly elements in $\mbC$.  It turns out that $\mbA$ is a composition algebra and any finite-dimensional composition algebra over $\mbC$ is isomorphic to one of these
$\mbA$ (see for example \cite[Chapter 5, Section 1]{VGO}).

We consider the following vector space of $\mbA$-Hermitian matrices of order $3$ with coefficients in $\mbA$:
$$
\SJ_3(\mbA) =\left\{ \begin{pmatrix} r_1 & \bar{x}_3 & \bar{x}_2 \\ x_3 & r_2 & \bar{x}_1 \\ x_2 & x_1 & r_3 \end{pmatrix},  r_i \in \mbC, x_i \in \mbA \right\}.
$$

It turns out that $\SJ_3(\mbA)$ has the structure of a Jordan algebra with multiplication given by $A \circ B = \frac{1}{2} (AB + BA)$, where $AB$ is the usual matrix multiplication.  The comatrix of $A \in \SJ_3(\mbA)$ is defined as
$$
\com(A) = A^2 - \tr(A) A + \tfrac{1}{2} ((\tr(A))^2 - \tr(A^2)) \Id.
$$
Then there exists a degree $3$ polynomial $\det(A)$ (called the determinant of $A$) such that $\com(A)\circ A = \det(A) \Id$. From this equality, we can easily deduce that
$$
\det(A) = \tfrac{1}{3} \tr(A^3) - \tfrac{1}{2} \tr(A) \tr(A^2) + \tfrac{1}{6}(\tr(A))^3.
$$

For $A = \begin{pmatrix} r_1 & \bar{x}_3 & \bar{x}_2 \\ x_3 & r_2 & \bar{x}_1 \\ x_2 & x_1 & r_3 \end{pmatrix}$, we have the following explicit formulae:

\begin{align*}
\tr(A)  = &\sum_i r_i, \quad  \quad  \tr(A^2) = \sum_i \left(r_i^2 + 2 x_i \bar{x}_i\right),  \\
\tr(A^3) = & \sum_i \left(r_i^3 + 3 \sum_{j\neq i} r_i x_j \bar{x}_j\right)  + \left(x_1x_3\bar{x}_2 + \bar{x}_2 x_1x_3 + x_3\bar{x}_2x_1 + x_2\bar{x}_3\bar{x}_1 + \bar{x}_3\bar{x}_1x_2 + \bar{x}_1x_2\bar{x}_3\right).
\end{align*}
It then follows that
\begin{align*}
\det(A)  = & r_1 r_2 r_3 - r_1 x_1 \bar{x}_1 - r_2 x_2 \bar{x}_2 - r_3 x_3 \bar{x}_3 \\ & + \tfrac{1}{3} (x_1x_3\bar{x}_2 + \bar{x}_2 x_1x_3 + x_3\bar{x}_2x_1 + x_2\bar{x}_3\bar{x}_1 + \bar{x}_3\bar{x}_1x_2 + \bar{x}_1x_2\bar{x}_3).
\end{align*}

Let us have a closer look at the $\mbC$-valued polynomial $\det$ on $\SJ_3(\mbA)$. For $A, B, C\in\SJ_3(\mbA)$,  define
\begin{align*}
& A\times B=\tfrac{1}{2}(2A\circ B-\tr(A)B-\tr(B)A+(\tr(A)\tr(B)-\tr(A\circ B))\Id), \\
&  (A, B, C)=\tr(A \circ(B \times C)).
\end{align*}
Then it follows that
$$
\com(A)=A\times A,\quad  \det(A)=\tfrac{1}{3}(A, A, A)\  {\text{ and }} \ \com(A)\times \com(A)=\det(A)A.
$$
We now define the following two subgroups of $\GL_\mbC(\SJ_3(\mbA))$:
\begin{align*}
\SL_3(\mbA) &= \{ g \in \GL_\mbC(\SJ_3(\mbA))| \det(g(A)) = \det (A), \forall A \in \SJ_3(\mbA)\}\\
\SO_3(\mbA) &= \{ g \in \SL_3(\mbA)| \tr(g(A)^2) = \tr (A^2), \forall A \in \SJ_3(\mbA)\}.
\end{align*}

The following table gives the corresponding groups:

\begin{center}
\begin{tabular}{|c | c  c  c  c| }
\hline
 $\mbA$ & $\quad \mbC$ & $\quad \mbC \oplus \mbC$ & $\quad \mbH_\mbC$ & $\quad \mbO_\mbC$ \\
\hline
$\SL_3(\mbA)$ & $\quad \SL_3$ & $\quad \SL_3 \times \SL_3 $ & $\quad \SL_6$ & $\quad E_6$ \\
\hline
$\SO_3(\mbA)$ & $\quad \SO_3$ & $\quad \SL_3$ & $\quad \Sp_6$ & $\quad F_4$\\
\hline
\end{tabular}
\end{center}

Consider the  two matrices
$$
M_{12} = \begin{pmatrix}  0 & 1 & 0 \\ 1 & 0 & 0 \\ 0 & 0 & 1 \end{pmatrix}, \quad  M_{23} = \begin{pmatrix}  1 & 0 & 0 \\ 0 & 0 & 1 \\ 0 & 1 & 0 \end{pmatrix}.
$$
Note that  $M_{12}^2=M_{23}^2=\Id$. Define
$$\sigma_{12}\colon \SJ_3(\mbA) \lra \SJ_3(\mbA), \quad \quad A \longmapsto M_{12} A M_{12}. $$
Similarly, we can define $\sigma_{23}$ by using $M_{23}$.

\begin{lem}\label{l.DefineSigma}
The two elements $\sigma_{12}, \sigma_{23}$ are in $\SO_3(\mbA)$, and the subgroup $\langle\sigma_{12},\sigma_{23}\rangle$ generated by them is isomorphic to $\mathfrak{S}_3$.
\end{lem}

\begin{proof}
For $A = \begin{pmatrix} r_1 & \bar{x}_3 & \bar{x}_2 \\ x_3 & r_2 & \bar{x}_1 \\ x_2 & x_1 & r_3 \end{pmatrix}$, we have
$$
\sigma_{12}(A) = \begin{pmatrix} r_2 & x_3 & \bar{x}_1 \\ \bar{x}_3 & r_1 & \bar{x}_2 \\ x_1 & x_2 & r_3 \end{pmatrix}, \quad \quad \sigma_{23}(A) = \begin{pmatrix} r_1 & \bar{x}_2 & \bar{x}_3 \\ x_2 & r_3 & x_1 \\ x_3 & \bar{x}_1 & r_2 \end{pmatrix}.
$$
Now it is straightforward to check that $\det(\sigma_{12}(A)) = \det (A)$ and $\tr(\sigma_{12}(A)^2) = \tr(A^2)$ by using previous explicit formulae; hence $\sigma_{12} \in \SO_3(\mbA)$. Similarly, we can show $\sigma_{23} \in \SO_3(\mbA)$.  By regarding $\sigma_{12}$ and $\sigma_{23}$ as the permutations $(12)$ and $(23)$ in $\mathfrak{S}_3$, respectively, we have $\langle\sigma_{12},\sigma_{23}\rangle=\mathfrak{S}_3\subset\SO_3(\mbA)$.
\end{proof}

\begin{rmk}\label{r.DefineSigma}
Note that $\sigma\cdot\diag(r_1, r_2, r_3)=\diag(r_{\sigma(1)}, r_{\sigma(2)}, r_{\sigma(3)})\in\SJ_3(\mbA)$ for any $\sigma\in\mathfrak{S}_3$ and any diagonal matrix $\diag(r_1, r_2, r_3)\in\SJ_3(\mbA)$.
\end{rmk}

\subsection{The involution on Lie algebras}

There exists an involution $\theta\colon \SL_3(\mbA) \to \SL_3(\mbA)$ coming from the symmetry of the Dynkin diagram of $\SL_3(\mbA)$, which satisfies $\SL_3(\mbA)^\theta = \SO_3(\mbA)$. The quotient $\SL_3(\mbA)/\SO_3(\mbA)$ is a symmetric homogeneous space.  The involution $\theta$ induces an involution (still denoted by $\theta$) on $\mathfrak{sl}_3(\mbA)$ whose fixed locus is $\mathfrak{so}_3(\mbA)$.  We have the following description of these Lie algebras:
\begin{align*}
\mathfrak{sl}_3(\mbA) &= \{ \phi \in \End(\sJ_3(\mbA))| (\phi(B), B, B) =0, \forall B \in \sJ_3(\mbA)\}. \\
\mathfrak{so}_3(\mbA) & =\{ \psi \in \End (\sJ_3(\mbA))| \psi(B \circ C) = \psi (B) \circ C + B \circ \psi(C), \forall B, C  \in \sJ_3(\mbA) \}.
\end{align*}

The following result is well known, but we include a proof here for the reader's convenience.

\begin{lem} \label{l.ThetaDecomposition}
Let $\SJ_3(\mbA)_0$ be the vector subspace of $\SJ_3(\mbA)$ consisting of traceless elements.
\begin{enumerate}
\item\label{l.TD-1} The map $\mu\colon \SJ_3(\mbA)_0 \to \End(\sJ_3(\mbA))$ given by $A \mapsto [B \mapsto 2A \circ B]$ embeds $\SJ_3(\mbA)_0$ into $\mathfrak{sl}_3(\mbA)$.

\item\label{l.TD-2} The involution $\theta$ acts on $\SJ_3(\mbA)_0$ by $-1$, and we have the decomposition into $\theta$-eigenvector spaces
$\mathfrak{sl}_3(\mbA) = \mathfrak{so}_3(\mbA) \oplus \SJ_3(\mbA)_0$.
\end{enumerate}
\end{lem}

\begin{proof}
For $A, B, C\in\SJ_3(\mbA)$, it is straightforward to show that
$$
\tr((A \circ B)\circ C) = \tr(A \circ (B \circ C)).
$$

Taking $A\in\SJ_3(\mbA)_0$ and $B\in\SJ_3(\mbA)$, one gets
\begin{eqnarray*}
 (A\circ B, B, B)=\tr((A\circ B)\circ\com(B))=\tr(A\circ (B\circ \com(B)))= \det(B)\tr(A)=0.
\end{eqnarray*}
It follows that $\mu(A)\in \mathfrak{sl}_3(\mbA)$. As $\mu(A)(\Id)=2A$, the map $\mu$ is injective. This shows~\eqref{l.TD-1}.

For~\eqref{l.TD-2}, we first check $\mu(\SJ_3(\mbA)_0) \cap \mathfrak{so}_3(\mbA) =0$.  Assume $\mu(A) \in \mathfrak{so}_3(\mbA)$; then $A \circ (B \circ C) = (A \circ B) \circ C + B \circ (A \circ C)$ for all $B, C$. We may take $B=C=\Id$, which shows that $A=0$.  By a dimension check, we get $\mathfrak{sl}_3(\mbA) = \mathfrak{so}_3(\mbA) \oplus \SJ_3(\mbA)_0$.  As the first part is the $\theta$-eigenspace of 1, the second part is the $\theta$-eigenspace of~$-1$.
\end{proof}

The subspace $\SJ_3(\mbA)_0$ is not necessarily a Lie subalgebra of $ \mathfrak{sl}_3(\mbA)$. On the other hand, we do have the following result.

\begin{lem}\label{l.T0torus}
  \leavevmode
  \begin{enumerate}
    \item\label{l.TO-1} The subset $T_0:=\{\diag(\lambda_1, \lambda_2, \lambda_3)\mid \lambda_1 \lambda_2 \lambda_3=1\}$ of $\SJ_3(\mbA)$ is a group under the Jordan algebra structure of $\SJ_3(\mbA)$, which is isomorphic as a group to the $2$-dimensional torus $(\mathbb{C}^*)^2$.

\item\label{l.TO-2} The map
\begin{alignat*}{2}
&\nu\colon  &T_0&\lra\SL_3(\mbA) \\
&& A&\longmapsto [B\longmapsto ABA]
\end{alignat*}
is an injective homomorphism of groups, and the associated homomorphism of Lie algebras is $\mu|_{\fh_0}\colon \fh_0\to \mathfrak{sl}_3(\mbA)$, where $\fh_0 = \{\diag(t_1, t_2, t_3)\mid t_1+t_2+t_3=0\}  \subset \SJ_3(\mbA)_0$.
In particular, \textit{via} $\nu$ and $\mu$ we can regard $T_0$ as a 2-dimensional torus of\, $\SL_3(\mbA)$ with Lie algebra $\fh_0\subset\mathfrak{sl}_3(\mbA)$.

\item\label{l.TO-3} Take any $\sigma\in\mathfrak{S}_3\subset\SO_3(\mbA)$. Then the inner automorphism and the adjoint representation of\, $\SL_3(\mbA)$ give rise to $\Inn_{\sigma}(T_0)=T_0$ and $\Ad_{\sigma}(\fh_0)=\fh_0$. More precisely,
\begin{alignat*}{2}
&\Inn_\sigma\colon &  T_0 &\longrightarrow T_0  \\
&& \diag(\lambda_1, \lambda_2, \lambda_3)&\longmapsto\diag(\lambda_{\sigma(1)}, \lambda_{\sigma(2)}, \lambda_{\sigma(3)}), \\
&\Ad_\sigma\colon & \fh_0&\longrightarrow\fh_0 \\
&& \diag(t_1, t_2, t_3)&\longmapsto\diag(t_{\sigma(1)}, t_{\sigma(2)}, t_{\sigma(3)}).
\end{alignat*}
\end{enumerate}
\end{lem}

\begin{proof}
  \begin{enumerate}[wide]
\item One has $A_1\circ A_2=A_1A_2=A_2A_1$ for $A_1, A_2\in T_0$. Hence $T_0$ is an abelian group. It follows that the group structure is the same as that of the $2$-dimensional torus.

\item Assume $A\in\ker(\nu)$; then $B= ABA$ for any $B \in  \SJ_3(A)$, which implies that
 $A^2=\Id$ and $BA=(ABA)A=AB$. This implies $A={\pm\rm Id}$. Since $A\in T_0$, we get $A=\Id$. The assertion on Lie algebras follows immediately.

\item For $A=\diag(\lambda_1, \lambda_2, \lambda_3)\in T_0$ (viewed as an element in $\SL_3(\mbA)$ \textit{via} $\nu$) and $B\in\SJ_3(\mbA)$, we have
  $$\Inn_{\sigma_{12}}(A) (B)=(\sigma_{12}A\sigma_{12})(B) =\sigma_{12}A(M_{12}BM_{12}) = M_{12}AM_{12}BM_{12}AM_{12}=\sigma_{12}(A)\cdot B.$$   It follows that $\Inn_{\sigma_{12}}(A)=\sigma_{12}(A)=\diag(\lambda_2, \lambda_1, \lambda_3)\in T_0$. Similarly, we have $\Inn_{\sigma_{23}}(A)=\sigma_{23}(A)=\diag(\lambda_1, \lambda_3, \lambda_2)\in T_0$. Consequently, $\Inn_{\sigma}(A)=\diag(\lambda_{\sigma(1)}, \lambda_{\sigma(2)}, \lambda_{\sigma(3)})\in T_0$ for any $\sigma\in\mathfrak{S}_3$.
\qedhere  \end{enumerate}
\end{proof}

\begin{lem} \label{l.CartanSA}
Let\, $\fl=\{v \in \mathfrak{sl}_3(\mbA) | [v, \fh_0] = 0 \}$  be the centralizer of\, $\fh_0$. Then there exists a Cartan subalgebra $\fh$ of $\mathfrak{sl}_3(\mbA)$ such that $\fh_0\subset\fh\subset\fl$, $\theta (\fh) = \fh$, $\fh_0 = \fh \cap \SJ_3(\mbA)_0$, and if moreover $\mbA\neq\mbC$, then $\fh \cap \mathfrak{so}_3(\mbA)$ is a Cartan subalgebra of $\mathfrak{so}_3(\mbA)$.
\end{lem}

\begin{proof}
By Lemma~\ref{l.T0torus}, $\fh_0$ is the Lie algebra of a torus in $\SL_3(\mbA)$. Since $\fh_0$ is $\theta$-stable, $\fl$ is $\theta$-stable. There is a direct sum decomposition into $\theta$-eigenspaces $\fl=(\fl\cap\mathfrak{so}_3(\mbA))\oplus(\fl\cap\SJ_3(\mbA)_0)$. We claim that $\fl\cap\SJ_3(\mbA)_0=\fh_0$. First assume this claim. Let $\fh_1$ be any Cartan subalgebra of $\fl\cap\mathfrak{so}_3(\mbA)$; then by the claim above, $\fh:=\fh_1\oplus\fh_0$ is a Cartan subalgebra of $\fl$ (hence  of $\mathfrak{sl}_3(\mbA)$). In particular, $\fh\cap\mathfrak{so}_3(\mbA)=\fh_1\simeq\fh/\fh_0$, and thus $\dim(\fh\cap\mathfrak{so}_3(\mbA))=\rank(\mathfrak{sl}_3(\mbA))-2$. When $\mbA\neq\mbC$, we have $\rank(\mathfrak{sl}_3(\mbA))=\rank(\mathfrak{so}_3(\mbA))+2$, and thus $\fh\cap\mathfrak{so}_3(\mbA)$ is a Cartan subalgebra of $\mathfrak{so}_3(\mbA)$. So $\fh$ is the required Cartan subalgebra of $\mathfrak{sl}_3(\mbA)$.

Now we turn to verifying the claim $\fl\cap\SJ_3(\mbA)_0=\fh_0$. Take $A\in \fl\cap\SJ_3(\mbA)_0$. Then for any $B\in\fh_0$ and $C\in \SJ_3(\mbA)$, we have $[\mu(A), \mu(B)](C)=0$, which implies that $D:=ABC+CBA-BAC-CAB=0$. We write
$$
A = \begin{pmatrix} r_1 & \bar{x}_3 & \bar{x}_2 \\ x_3 & r_2 & \bar{x}_1 \\ x_2 & x_1 & r_3 \end{pmatrix}.
$$
Take $B=\diag(b_1, b_2, b_3)\in\fh_0$ and  $C=\diag(c_1, c_2, c_3)\in \SJ_3(\mbA)$. Then $$D = \begin{pmatrix} 0 &(b_1-b_2)(c_1-c_2) \bar{x}_3 & (b_1-b_3)(c_1-c_3)\bar{x}_2 \\ (b_1-b_2)(c_1-c_2)x_3 & 0 & (b_2-b_3)(c_2-c_3)\bar{x}_1 \\ (b_1-b_3)(c_1-c_3)x_2 & (b_2-b_3)(c_2-c_3)x_1 & 0 \end{pmatrix}.$$ Varying $B$ and $C$, we get $x_1=x_2=x_3=0$ and $A=\diag(r_1, r_2, r_3)$. As $A\in\SJ_3(\mbA)_0$, one has $r_1+r_2+r_3=0$ and $A\in\fh_0$, which concludes the proof.
\end{proof}

\begin{lem} \label{l.Cartan2}
Let $\fh_1$ be a Cartan subalgebra of $\mathfrak{so}_3(\mbA)$, and let $\fh=\{v \in \mathfrak{sl}_3(\mbA) | [v, \fh_1] = 0 \}$ be the centralizer of $\fh_1$ in $\mathfrak{sl}_3(\mbA)$. Then $\fh$ is a Cartan subalgebra of $\mathfrak{sl}_3(\mbA)$, $\fh_1 \subset \fh$, $\theta(\fh)=\fh$, $\fh_1 = \fh \cap \mathfrak{so}_3(\mbA)$ and
$$\dim(\fh/\fh_1) =
\begin{cases}   1  & \text{if} \  \mbA = \mbC, \\
2 & \text{otherwise}.
\end{cases}
$$
\end{lem}

\begin{proof}
Since $\fh_1$ is $\theta$-stable, its centralizer $\fh=\{v \in \mathfrak{sl}_3(\mbA) | [v, \fh_1] = 0 \}$ is $\theta$-stable. Then there is a direct sum decomposition into $\theta$-eigenspaces $\fh=(\fh\cap\mathfrak{so}_3(\mbA))\oplus(\fh\cap\SJ_3(\mbA)_0)$. Since $\fh\cap\mathfrak{so}_3(\mbA)=\fh_1$, $\theta$ acts on the quotient algebra $\fh/\fh_1$ by $-1$. Let $\bar{\fk}$ be a Cartan subalgebra of $\fh/\fh_1$. Then its preimage $\fk$ in $\mathfrak{sl}_3(\mbA)$ is a Cartan subalgebra of $\mathfrak{sl}_3(\mbA)$ satisfying $\fh_1 \subset \fk$, $\theta(\fk)=\fk$ and $\fh_1 = \fk \cap \mathfrak{so}_3(\mbA)$. Since $\fk$ and $\fh_1$ are Cartan subalgebras of $\mathfrak{sl}_3(\mbA)$ and $\mathfrak{so}_3(\mbA)$, respectively, $\dim(\fk/\fh_1)=\rank(\mathfrak{sl}_3(\mbA))-\rank(\mathfrak{so}_3(\mbA))$. Then $\fk/\fh_1$ has the dimension as stated. It remains to prove that $\fh=\fk$.

The space $\SJ_3(\mbA)_0$ is an irreducible module of $\SO_3(\mbA)$. More precisely, we have
\begin{center}
\begin{tabular}{|c | c  c  c  c| }
\hline
 $\mbA$ & $\quad \mbC$ & $\quad \mbC \oplus \mbC$ & $\quad \mbH_\mbC$ & $\quad \mbO_\mbC$ \\
\hline
type of $\SO_3(\mbA)$ & $\quad A_1$ & $\quad  A_2$ & $\quad C_3$ & $\quad F_4$\\
\hline
highest weight of $\SJ_3(\mbA)_0$ & $\quad 2\omega_1$ & $\quad \omega_1+\omega_2 $ & $\quad \omega_2$ & $\quad \omega_1$ \\
\hline
\end{tabular}
\end{center}
The table  can be deduced from  \cite[Theorem 3 and Lemma 17]{R10}.
A direct calculation shows that the multiplicity of weight zero in the irreducible $\mathfrak{so}_3(\mbA)$-module $\SJ_3(\mbA)_0$ is $1$ in the case $\mbA=\mbC$ and $2$ in other cases. Since $\fh\cap\SJ_3(\mbA)_0$ is the $\fh_1$-eigenspace of weight zero, $\dim(\fh/\fh_1)=\dim(\fh\cap\SJ_3(\mbA)_0)=\dim(\fk/\fh_1)$. Hence $\fk=\fh$, completing the proof.
\end{proof}

\begin{prop} \label{p.MaxTorus}
  \leavevmode
  \begin{enumerate}
    \item\label{p.MT-1} Given a  maximal torus $T_1$ of\, $\SO_3(\mbA)$, the identity component $T$ of its centralizer is a maximal torus of $\SL_3(\mbA)$.

\item\label{p.MT-2} If $\mbA \neq \mbC$, there exists a $g \in \SO_3(\mbA)$ such that the conjugate $T^g:= g T g^{-1}$ satisfies $ T_0 \subset T^g$ and
  $T_0 \simeq T^g/(T^g \cap \SO_3(\mbA))$, where $T_0$ is as in Lemma~\ref{l.T0torus}.
  \end{enumerate}
\end{prop}

\begin{proof}
Claim~\eqref{p.MT-1} follows immediately from Lemma~\ref{l.Cartan2}. By Lemma~\ref{l.CartanSA}, there is a maximal torus $T'$ of $\SL_3(\mbA)$ such that $T_0\subset T'$, the identity component $T'_1$ of $T'\cap\SO_3(\mbA)$ is a maximal torus of $\SO_3(\mbA)$, and $T_0 \simeq T'/(T' \cap \SO_3(\mbA))$. Both $T_1$ and $T'_1$ are maximal tori of $\SO_3(\mbA)$, so there exists a $g\in\SO_3(\mbA)$ such that $T'_1=gT_1g^{-1}$. By Lemma~\ref{l.Cartan2} (resp.\ by the choice of $T$), the maximal torus $T'$ (resp.\ $T^g:= g T g^{-1}$) is the identity component of the centralizer of $T'_1$ (resp.\ $gT_1g^{-1}$). As $T'_1=gT_1g^{-1}$, we have $T'=T^g$, which concludes the proof.
\end{proof}

\subsection{The symmetric variety $\boldsymbol{X(\mbA)}$}

We consider the following rational map:
\begin{align*}
\Phi\colon \mbP(\mbC \oplus \SJ_3(\mbA)) & \longdashrightarrow \mbP(\mbC \oplus \SJ_3(\mbA) \oplus \SJ_3(\mbA) \oplus \mbC) \\  [t: A] & \longmapsto [t^3: t^2A: t\ \com(A): \det(A)].
\end{align*}

We denote by $G_\omega(\mbA^3, \mbA^6)$ the closure of the image of $\Phi$, which turns out to be a rational homogeneous space corresponding to the third row in  Freudenthal's magic square of varieties (\textit{cf.} \cite{LM}).

By \cite[Proposition 4.1]{LM}, the action of $\SL_3(\mbA)$ on $\mbP(\SJ_3(\mbA))$ has a unique closed orbit, denoted by $\mbA\mbP^2$, which is just one of the four Severi varieties.  Note that  $\mbA\mbP^2$ is also the variety of lines on $G_\omega(\mbA^3, \mbA^6)$  through a fixed point.

The following table collects information about all these varieties:

\begin{center}
\begin{tabular}{|c | c  c  c  c| }
\hline
 $\mbA$ & $\quad \mbC$ & $\quad \mbC \oplus \mbC$ & $\quad \mbH_\mbC$ & $\quad \mbO_\mbC$ \\
\hline
$G_\omega(\mbA^3, \mbA^6)$ & $\quad \Lag(3,6)$ & $\quad \Gr(3,6)$ & $\quad \mbS_6$ & $\quad E_7/P_7$ \\
\hline
$\mbA\mbP^2$ & $\quad \nu_2(\mbP^2)$ & $\quad \mbP^2 \times \mbP^2$ & $\quad \Gr(2,6)$ & $\quad E_6/P_1$ \\
\hline
\end{tabular}
\end{center}

Let $X(\mbA)$ be the closure of the image under $\Phi$ of the cubic hypersurface $t^3 = \det(A)$ in $ \mbP(\mbC \oplus \SJ_3(\mbA))$, which is a hyperplane  section of $G_\omega(\mbA^3, \mbA^6)$.
We call $X(\mbA)$ the symmetric manifold associated to the composition algebra $\mbA$.

We can now summarize some properties of $X(\mbA)$ as follows.

\begin{prop}\label{p.X(A)properties}\leavevmode
  \begin{enumerate}
    \item\label{p.X(A)p-1} The variety $X(\mbA)$ is the smooth equivariant completion of\, $\SL_3(\mbA)/\SO_3(\mbA)$ of Picard number 1.

 \item\label{p.X(A)p-2} The connected automorphism group of\, $X(\mbA)$ is isomorphic to $\SL_3(\mbA)$ up to isogeny, and the involution $\theta$ of\, $\SL_3(\mbA)$ induces an involution of\, $X(\mbA)$, denoted by $\theta$ again.

\item\label{p.X(A)p-3} The variety $X(\mbA)$ is locally rigid; \textit{i.e.} $H^1(X(\mbA), T_{X(\mbA)})=0$.

\item\label{p.X(A)p-4} The variety of lines through a general point of\, $X(\mbA)$ is a smooth hyperplane section of $\mbA \mbP^2$, which is respectively $\nu_2(\mbQ^1)$, $\mbP T^*_{\mbP^2}$, $\Gr_\omega(2, 6)$, $F_4/P_1$, where $\Gr_\omega(2, 6)$ is the symplectic Grassmannian.
\end{enumerate}
\end{prop}

\begin{proof}
 Claim~\eqref{p.X(A)p-1} follows from \cite[Lemma 17]{R10}, and claim~\eqref{p.X(A)p-2} is from \cite[Theorems 2 and 3]{R10}. Claim~\eqref{p.X(A)p-3} follows from \cite[Theorem 1.1]{BFM}.
 As $X(\mbA)$ is a smooth hyperplane section of $G_\omega(\mbA^3, \mbA^6)$, its variety of lines through a general point is a smooth hyperplane section of that of $G_\omega(\mbA^3, \mbA^6)$, namely a smooth hyperplane section of
 $\mbA\mbP^2$, which gives claim~\eqref{p.X(A)p-4}.
\end{proof}

\begin{lem}\label{l.ToricFiber}
Let $o=[1: \Id] \in \mbP(\mbC \oplus \SJ_3(\mbA))$.
\begin{enumerate}
\item\label{l.TF-1} We have $\SL_3(\mbA) \cdot o \simeq \SL_3(\mbA)/\SO_3(\mbA)$.
\item\label{l.TF-2} The image closure of\, $\Phi(T_0 \cdot o)$, denoted by $Y(\mbA)$, is isomorphic to the blowup of\, $\BP^2$ along its three coordinate points.
\end{enumerate}
\end{lem}

\begin{proof}
Claim \eqref{l.TF-1} is from \cite{R10}. For~\eqref{l.TF-2}, recall that the Lie group $T_0$ acts on $J_3(\mbA)$ by $A \mapsto (B \mapsto ABA)$ (\textit{cf.} Lemma~\ref{l.T0torus}\eqref{l.TO-2}). Write $A:=\begin{psmallmatrix} \lambda_1 & 0 & 0 \\ 0 & \lambda_2 & 0 \\ 0 & 0 & \lambda_3 \end{psmallmatrix}$ in $T_0$; then $A\cdot \Id = \begin{psmallmatrix} \lambda_1^2 & 0 & 0 \\ 0 & \lambda_2^2 & 0 \\ 0 & 0 & \lambda_3^2 \end{psmallmatrix} $.  It follows that the image closure of $\Phi(T_0 \cdot o)$ is the closure of the elements
$$
[1: \lambda_1^2: \lambda_2^2: \lambda_3^2: \lambda_1^{-2}:   \lambda_2^{-2}: \lambda_3^{-2}: 1], \quad  \lambda_1 \lambda_2 \lambda_3=1.
$$
It is easy to see that this is  the blowup of $\BP^2$ along its three coordinate points.
\end{proof}

\subsection{The  $\boldsymbol{\mathfrak{S}_3}$-action on the  toric surface  $\boldsymbol{Y(\mbA)}$}

Let $M_1=\mathbb{P}^2$ with three coordinate points $P_1=[1:0:0]$, $P_2=[0:1:0]$, $P_3=[0:0:1]$.  By Lemma~\ref{l.ToricFiber}, $Y(\mbA)$ is the blowup of $M_1$ at $\{P_1, P_2, P_3\}$.  Let $D_i\subset M_1$, $i=1,2,3$,  be the lines through the points $P_{j_1}$ and $P_{j_2}$ such that $\{i,j_1,j_2\}=\{1,2,3\}$.  Let $E_i\subset Y(\mbA)$ be the exceptional divisors over the points $P_i$, $i=1,2,3$.  By abusing notation, we again denote  by $D_i\subset Y(\mbA)$, $i=1,2,3$, the strict transform of the line $D_i\subset M_1$. Let $M_2$ be the blowdown of $D_i\subset Y(\mbA)$. Then $M_2=\mathbb{P}^2$.

\begin{prop} \label{p.S3action}
The subgroup $\mathfrak{S}_3$ of\, $\SO_3(\mbA)$, which is introduced in Lemma~\ref{l.DefineSigma}, stabilizes the open torus $T_0\cdot o\subset Y(\mbA)\subset X(\mbA)$.  The action of\, $\mathfrak{S}_3$ on the boundary divisors of\, $Y(\mbA)$ is as follows: $\sigma(D_i)=D_{\sigma(i)}$ and $\sigma(E_j)=E_{\sigma(j)}$ for all $\sigma\in \mathfrak{S}_3$ and $1\leq i,j\leq 3$.
\end{prop}

\begin{proof}
Take $\sigma\in \mathfrak{S}_3\subset\SO_3(\mbA)$ and $B\in T_0$. Then $\sigma\cdot (B\cdot o)=\Inn_\sigma(B)\cdot (\sigma\cdot o)$. Note that $\sigma\cdot o=o$ since $\sigma\in\SO_3(\mbA)$, and $\Inn_\sigma(B)\in T_0$ by Lemma~\ref{l.T0torus}. Hence $\sigma\cdot (B\cdot o)\in T_0\cdot o$, which shows the first claim.

By the proof of Lemma~\ref{l.ToricFiber}, the composition of $T_0$-equivariant maps $T_0\rightarrow T_0\cdot o\rightarrow M_1$ is $\diag(\lambda_1, \lambda_2, \lambda_3)\mapsto [\lambda_1^2:\lambda_2^2: \lambda_3^2]$. The action of $\mathfrak{S}_3$ on $T_0$ extends to $M_1$ by permuting coordinates.  This action of $\mathfrak{S}_3$ lifts to $Y(\mbA)$ since the blowup center is $\mathfrak{S}_3$-stable. Furthermore, the lifting coincides with the restriction to $Y(\mbA)$ of the $\mathfrak{S}_3$-action on $X(\mbA)$ because these two actions coincide on the open torus orbit of $Y(\mbA)$.
Note that $\sigma$ sends $P_i$ to $P_{\sigma(i)}$ and sends the line joining $P_i$ and $P_j$ to the line joining $P_{\sigma(i)}$ and $P_{\sigma(j)}$. The conclusion follows.
\end{proof}

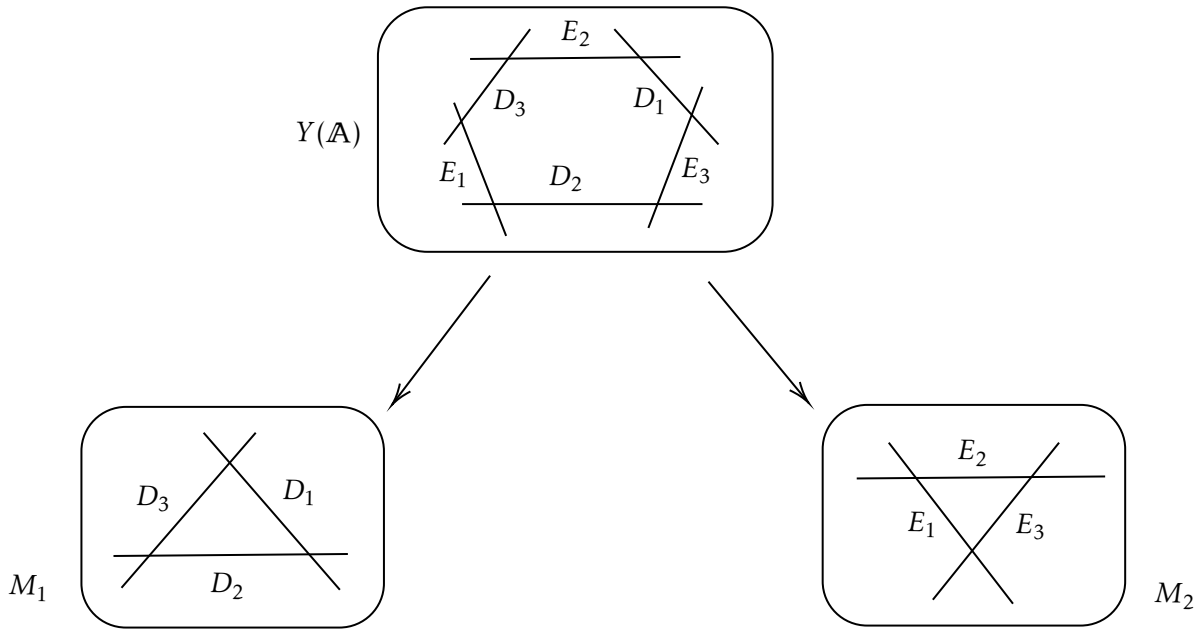
\tikzset{every picture/.style={line width=0.75pt}} 

\begin{tikzpicture}[x=0.75pt,y=0.75pt,yscale=-1,xscale=1]

\draw   (202,527.45) .. controls (202,513.86) and (213.01,502.85) .. (226.6,502.85) -- (374.4,502.85) .. controls (387.99,502.85) and (399,513.86) .. (399,527.45) -- (399,601.25) .. controls (399,614.84) and (387.99,625.85) .. (374.4,625.85) -- (226.6,625.85) .. controls (213.01,625.85) and (202,614.84) .. (202,601.25) -- cycle ;
\draw   (54,726.05) .. controls (54,713.79) and (63.94,703.85) .. (76.2,703.85) -- (182.8,703.85) .. controls (195.06,703.85) and (205,713.79) .. (205,726.05) -- (205,792.65) .. controls (205,804.91) and (195.06,814.85) .. (182.8,814.85) -- (76.2,814.85) .. controls (63.94,814.85) and (54,804.91) .. (54,792.65) -- cycle ;
\draw   (424,725.05) .. controls (424,712.79) and (433.94,702.85) .. (446.2,702.85) -- (552.8,702.85) .. controls (565.06,702.85) and (575,712.79) .. (575,725.05) -- (575,791.65) .. controls (575,803.91) and (565.06,813.85) .. (552.8,813.85) -- (446.2,813.85) .. controls (433.94,813.85) and (424,803.91) .. (424,791.65) -- cycle ;
\draw    (115,716.85) -- (183,795.85) ;
\draw    (70,778.85) -- (187,777.85) ;
\draw    (141,716.85) -- (74,794.85) ;
\draw [color={rgb, 255:red, 0; green, 0; blue, 0 }  ,draw opacity=1 ]   (248,528.85) -- (353,527.85) ;
\draw    (320,513.85) -- (372,571.85) ;
\draw    (278,513.85) -- (235,571.85) ;
\draw    (244,601.85) -- (364,601.85) ;
\draw [color={rgb, 255:red, 0; green, 0; blue, 0 }  ,draw opacity=1 ]   (240,550.85) -- (266,617.85) ;
\draw [color={rgb, 255:red, 0; green, 0; blue, 0 }  ,draw opacity=1 ]   (364,542.85) -- (337,613.85) ;
\draw [color={rgb, 255:red, 0; green, 0; blue, 0 }  ,draw opacity=1 ]   (457,721.85) -- (519,802.85) ;
\draw [color={rgb, 255:red, 0; green, 0; blue, 0 }  ,draw opacity=1 ]   (441,739.85) -- (565,738.85) ;
\draw [color={rgb, 255:red, 0; green, 0; blue, 0 }  ,draw opacity=1 ]   (542,721.85) -- (479,800.85) ;
\draw    (367,640.85) -- (415.73,700.3) ;
\draw [shift={(417,701.85)}, rotate = 230.66] [color={rgb, 255:red, 0; green, 0; blue, 0 }  ][line width=0.75]    (10.93,-3.29) .. controls (6.95,-1.4) and (3.31,-0.3) .. (0,0) .. controls (3.31,0.3) and (6.95,1.4) .. (10.93,3.29)   ;
\draw    (258,637.85) -- (210.2,701.25) ;
\draw [shift={(209,702.85)}, rotate = 307.01] [color={rgb, 255:red, 0; green, 0; blue, 0 }  ][line width=0.75]    (10.93,-3.29) .. controls (6.95,-1.4) and (3.31,-0.3) .. (0,0) .. controls (3.31,0.3) and (6.95,1.4) .. (10.93,3.29)   ;

\draw (90,750) node  [color={rgb, 255:red, 0; green, 0; blue, 0 }  ,opacity=1 ,rotate=-1.12]  {$D_{3}$};
\draw (127,794) node  [color={rgb, 255:red, 0; green, 0; blue, 0 }  ,opacity=1 ,rotate=-1.12]  {$D_{2}$};
\draw (163,747) node  [color={rgb, 255:red, 0; green, 0; blue, 0 }  ,opacity=1 ,rotate=-1.12]  {$D_{1}$};
\draw (301,516) node  [color={rgb, 255:red, 0; green, 0; blue, 0 }  ,opacity=1 ,rotate=-1.12]  {$E_{2}$};
\draw (338,550) node  [color={rgb, 255:red, 0; green, 0; blue, 0 }  ,opacity=1 ,rotate=-1.12]  {$D_{1}$};
\draw (296,589) node  [color={rgb, 255:red, 0; green, 0; blue, 0 } ,opacity=1 ,rotate=-1.12]  {$D_{2}$};
\draw (268,551) node  [color={rgb, 255:red, 0; green, 0; blue, 0 }  ,opacity=1 ,rotate=-1.12]  {$D_{3}$};
\draw (240,587) node  [color={rgb, 255:red, 0; green, 0; blue, 0 }  ,opacity=1 ,rotate=-1.12]  {$E_{1}$};
\draw (361,585) node  [color={rgb, 255:red, 0; green, 0; blue, 0 }  ,opacity=1 ,rotate=-1.12]  {$E_{3}$};
\draw (499,727) node  [color={rgb, 255:red, 0; green, 0; blue, 0 }  ,opacity=1 ,rotate=-1.12]  {$E_{2}$};
\draw (528,764) node  [color={rgb, 255:red, 0; green, 0; blue, 0 }  ,opacity=1 ,rotate=-1.12]  {$E_{3}$};
\draw (474,763) node  [color={rgb, 255:red, 0; green, 0; blue, 0 } ,opacity=1 ,rotate=-1.12]  {$E_{1}$};
\draw (28,795) node  [color={rgb, 255:red, 0; green, 0; blue, 0 } ,opacity=1 ,rotate=-1.12]  {$M_{1}$};
\draw (600,799) node  [color={rgb, 255:red, 0; green, 0; blue, 0 }  ,opacity=1 ,rotate=-1.12]  {$M_{2}$};
\draw (177,567) node  [color={rgb, 255:red, 0; green, 0; blue, 0 }  ,opacity=1 ,rotate=-1.12]  {$Y(\mbA)$};
\end{tikzpicture}

\medskip
Now we study the involution $\theta$ on $X(\mbA)$.

\begin{prop}\label{p.involution-Y(A)}
The involution $\theta$ of $X(\mbA)$ stabilizes $T_0\cdot o\subset Y(\mbA)\subset X(\mbA)$. Furthermore, $\theta(D_i)=E_i$ and $\theta(E_i)=D_i$ for $i=1,2,3$.
\end{prop}

\begin{proof}
  Let $\vartheta$ be the involution on $M_1\times M_2$ given by
  \smash{$\vartheta([x_1:x_2:x_3],\mkern-1mu [y_1:y_2:y_3])=([y_1:y_2:y_3],\mkern-1mu [x_1:x_2:x_3])$.}
  By the proof of Lemma~\ref{l.ToricFiber}, we have $\vartheta(b\cdot o)=b^{-1}\cdot o$ for each $b\in T_0$.  Hence, $\vartheta$ stabilizes $T_0\cdot o$ and its closure $Y(\mbA)$. Furthermore, by the definition of $\vartheta$, we have $\vartheta(D_i)=E_i$ and $\vartheta(E_i)=D_i$ for $i=1,2,3$.

It remains to verify that $\theta$ coincides with the restriction of $\vartheta$ on $T_0\cdot o$. The involution $\theta$ of $\mathfrak{sl}_3(\mbA)$ fixes $\mathfrak{so}_3(\mbA)$ and acts on $\SJ_3(\mbA)_0$ by $-1$. In particular, $\theta(\xi)=-\xi$ and $\theta(b)=b^{-1}$ for $\xi\in\fh_0$ and $b\in T_0$, proving the claim.
 \end{proof}

\begin{lem} \label{l.D-E-lin-equivalence}
The Picard group of $Y(\mbA)$ is generated by $ D_1$, $D_2$, $D_3$, $E_1$, $E_2$ and $E_3$. Moreover, we have the following rational equivalence relations:
$$
 D_1-E_1 \equiv_{\lin} D_2-E_2 \equiv_{\lin} D_3-E_3.
$$
\end{lem}

\begin{proof}
Since $Y(\mbA)$ is a toric variety, its Picard group is generated by prime boundary divisors $ D_1$, $D_2$, $D_3$, $E_1$, $E_2$ and $E_3$. The line joining $P_1$ and $P_2$ and the line joining $P_1$ and $P_3$ are linearly equivalent in $M_1$. Pulling back to $Y(\mbA)$, one gets $E_1+E_2+D_3\equiv_{\lin} E_1+E_3+D_2$, which implies $D_2-E_2\equiv_{\lin} D_3-E_3$. Similarly, one gets $D_1-E_1\equiv_{\lin} D_2-E_2$.
\end{proof}

\begin{prop} \label{p.Picard-invariant}
Let $\Gamma \subset\Aut(Y(\mbA))$ be a subgroup. Denote by $\Pic(Y(\mbA))^\Gamma$ the invariant subgroup of\, $\Pic(Y(\mbA))$ under the action of\, $\Gamma$. For a finite order element $\sigma \in \Aut(Y(\mbA))$, we set $\Pic(Y(\mbA))^\sigma = \Pic(Y(\mbA))^{<\sigma>}$. Then
\begin{itemize}
\item $\Pic(Y(\mbA))$ is a free abelian group of rank 4 with basis $\{D_1,E_1,E_2,E_3\}$;
    \item $\Pic(Y(\mbA))^{\sigma_{12}}$ is a free abelian group of rank 3 with basis
    $\{D_1+E_2,E_1+E_2,E_3\}$;
    \item $\Pic(Y(\mbA))^{\sigma_{13}}$ is a free abelian group of rank 3 with basis
        $\{D_1+E_3,E_1+E_3,E_2\}$;
    \item $\Pic(Y(\mbA))^{\sigma_{23}}$ is a free abelian group of rank 3 with basis $\{ D_1,E_1,E_2+E_3\}$;
    \item $\Pic(Y(\mbA))^{\sigma_{123}}=\Pic(Y(\mbA))^{\sigma_{321}}=\Pic(Y(\mbA))^{\mathfrak{S}_3}$ is a free abelian group of rank $2$ with basis $\{D_1+E_2+E_3,E_1+E_2+E_3\}$.
\end{itemize}
\end{prop}

\begin{proof}
As $Y(\mbA)$ is the blowup of $\mathbb{P}^2$ along three points, $\Pic(Y(\mbA))$ is freely generated by $E_1$, $E_2$, $E_3$ and the pullback of a line in $\mathbb{P}^2$, namely $D_1+E_2+E_3$, which gives the first claim.

For any divisor $aD_1+bE_1+cE_2+dE_3\in \Pic(Y(\mbA))^{\sigma_{12}}$, where $a,b,c,d\in \mathbb{Z}$,
$$\begin{array}{rcl}
     \sigma_{12}(aD_1+bE_1+cE_2+dE_3)&
     =& aD_2+bE_2+cE_1+dE_3\\
     &=&a(E_2-E_1+D_1)+bE_2+cE_1+dE_3\\
     &=&aD_1+(c-a)E_1+(a+b)E_2+dE_3.
\end{array}$$
Then $c=a+b$ and $$aD_1+bE_1+cE_2+dE_3=a(D_1+E_2)+b(E_1+E_2)+dE_3.$$
It follows that $\Pic(Y(\mbA))^{\sigma_{12}}$ is of rank 3 and  $D_1+E_2,E_1+E_2,E_3$ is a $\mathbb{Z}$-basis.

The proofs for other claims are similar.
\end{proof}

\section{Proof of Theorem~\ref{t.main} via reduction to a family of surfaces}

\subsection{Invariance of varieties of minimal rational tangents}

For a uniruled projective manifold $X$, let $\RatCurves^\n(X)$ denote the normalization of the space of rational curves on $X$ (see \cite[Proposition~II.2.11]{Kollar}). Every irreducible component $\sK$ of $\RatCurves^\n(X)$ is a (normal) quasi-projective variety equipped with a quasi-finite morphism to the Chow variety of $X$; the image consists of the Chow points of irreducible, generically reduced rational curves.  There is a universal family $\sU$ with projections $\upsilon \colon \sU \to \sK$, $\mu \colon \sU \to X$, and $\upsilon$ is a $\bP^1$-bundle (for these results, see \cite[Proposition~II.2.11 and Theorem~II.2.15]{Kollar}).

For any $x \in X$, let $\sU_x := \mu^{-1}(x)$ and $\sK_x := \upsilon(\sU_x)$. We call $\sK$ a \emph{family of minimal rational curves} if $\sK_x$ is non-empty and projective for a general point $x$.  There is a rational map $\iota_x\colon \sK_x \dasharrow \bP T_x X$ (the projective space of lines in the tangent space at $x$) that sends any curve which is smooth at $x$ to its tangent direction.  The closure of the image of $\iota_x$ is denoted by $\sC_x$ and called the \emph{variety of minimal rational tangents} (VMRT) at the point $x$.  By \cite[Theorem~1]{HM2} and \cite[Theorem~3.4]{Kebekus}, composing $\iota_x$ with the normalization map $\sK_x^\n \to \sK_x$ yields the normalization of $\sC_x$. Also, $\sK_x^\n$ is a union of components of the variety $\RatCurves^\n(x, X)$ defined in \cite[II.(2.11.2)]{Kollar} and hence is smooth for $x \in X$ general by \cite[Corollary~II.3.11.5]{Kollar}. In this case, $\sU_x \simeq \sK_x^\n$ is smooth, and the rational map $\iota_x$ induces a birational morphism $\sU_x \simeq \sK_x^\n \to \sC_x$, which is still denoted by $\iota_x$ by abuse of notation. Since $\mcU_x$ is both the normalization of $\sK_x$ and that of $\mcC_x$, we call $\mcU_x$ the normalized Chow space or the normalized VMRT.

By Proposition~\ref{p.X(A)properties}, the variety $X(\mbA)$ is covered by lines, and its VMRT at a general point is just the variety of lines through that point, denoted by $\sC(\mbA) \subset \mbP(V_\mbA)$ with $V_\mbA$ being the tangent space of $X(\mbA)$ at a general point, which is respectively $\nu_2(\mbQ^1)$, $\mbP T^*_{\mbP^2}$, $\Gr_\omega(2, 6)$, $F_4/P_1$ with the natural embedding.

 For a family of smooth projective varieties $\pi\colon \mcX \to \Delta$ with $\mcX_t \simeq X(\mbA)$ for all $t \neq 0$, we take a general section $\tau\colon \Delta \to \mcX$ such that $\tau(t)$ is a general point in $\mcX_t$ for all $t \in \Delta$.  By considering the VMRT of $\mcX_t$ at $\tau(t)$, we get a family of embedded projective subvarieties with general fibers isomorphic to $\sC(\mbA) \subset \mbP(V_\mbA)$.

We first prove the invariance of the VMRT, which means that the central fiber has the same VMRT as the general fiber. Note that the cases when $\mbA=\mbH_\mbC$ or $\mbO_\mbC$ are proved in \cite{KP}. For the reader's convenience, we include the proof for all three cases here.

\begin{prop}\label{p.Invariance}
Assume $\mbA \neq \mbC$.  Consider a family of smooth projective varieties $\pi\colon \mcX \to \Delta$ with $\mcX_t \simeq X(\mbA)$ for all $t \neq 0$. Then the VMRT of $\mcX_0$ at a general point is projectively equivalent to the VMRT of $X(\mbA)$ at a general point.
\end{prop}

\begin{proof}
Firstly we show that the normalized Chow space $\mcU_{x_0}$ of $\mcX_0$ at a general point $x_0$ is isomorphic to $\sC(\mbA)$. Take a general section $t\in\Delta\mapsto x_t\in\mcX_t$ of $\pi$ passing through the general point $x_0$ in $\mcX_0$. Shrinking $\Delta$ if necessary, we can assume that $x_t$ is general in $\mcX_t$ for each $t\neq 0$. The normalized Chow spaces $\mcU_{x_t}$ along this section give a family of smooth projective varieties such that $\mcU_{x_t} \simeq \sC(\mbA)$ for $t \neq 0$. If $\mbA \neq \mbC \oplus \mbC$, then $\mcU_{x_0} \simeq \sC(\mbA)$ by Theorem~\ref{t.HMRigidity}.  Now assume $\mbA = \mbC \oplus \mbC$. Consider the normalization map $\iota_{x_t}\colon \mcU_{x_t} \to \mcC_{x_t} \subset \mbP T_{x_t}\mcX_t$. Note that for any $t \neq 0$, we have $\mcU_{x_t} \simeq \mbP T^*_{\mbP^2}$, thus $K_{\mcU_{x_t}}+\iota_{x_t}^*\mcO(2)=0$. It follows that this equality also holds for $t=0$, which implies that $\mcU_{x_0}$ is a Fano threefold with index $2$ and we can apply
\cite[Theorem]{W} to deduce that $\mcU_{x_0} \simeq \mbP(T^*_{\mbP^2})$ (note that in \cite{W}, the projectivisation is taken in the sense of Grothendieck).

Recall that for a smooth projective subvariety $Z \subset \BP V$, the variety of tangential lines of $Z$ is the subvariety $\mcT_Z \subset \Gr(2, V) \subset \BP (\wedge^2 V)$ consisting of the tangential lines of $Z$. By \cite[Lemma 2.12]{FL}, the tangential variety of $\mbP(T^*_{\mbP^2})$ is non-degenerate. By \cite[Proposition 2.6]{Hw01}, the tangential varieties of $\Gr_\omega(2, 6)$ and $F_4/P_1$ are both non-degenerate. Now we can use the same argument as that of \cite[Proposition 3.9]{FL} to conclude the proof.
\end{proof}

Recall that a vector group of dimension $g$ is the additive group $\mathbb{G}_a^g$.  An equivariant compactification of $\mathbb{G}_a^g$ is a smooth projective $\mathbb{G}_a^g$-variety $Y$ which admits an open $\mathbb{G}_a^g$-orbit $O$ isomorphic to $\mathbb{G}_a^g$.  The boundary $\partial Y = Y \setminus O$ is a union of irreducible reduced divisors $\cup_j E_j$.  It follows that $\Pic(Y) = \oplus_{j} \mathbb{Z} [E_j]$. Moreover, we have $-K_Y = \sum_{j}a_j E_j$ with $a_j \geq 2$ by \cite[Theorem 2.7]{HT}. In particular, the support of $-K_Y$ is the whole boundary of $Y$.

By Proposition~\ref{p.X(A)properties}, we have $\mathfrak{aut}(X(\mbA)) = \mathfrak{sl}_3(\mbA)$.  Note that $\mathfrak{aut}(\sC(\mbA)) = \mathfrak{so}_3(\mbA)$, which implies that
\begin{equation} \label{eq}
\dim \mathfrak{aut}(X(\mbA)) = \dim \mathfrak{aut}(\sC(\mbA)) + \dim X(\mbA).
\end{equation}

For a specialization $\pi\colon \mcX \to \Delta$ of $X(\mbA)$, equality~\eqref{eq} implies the following precise information on the central fiber by the proofs of \cite[Lemma 4.6 and Corollary 4.8]{FL}. This result extends \cite[Theorem 1.1]{KP} to $\mbA \neq \mbC$.

\begin{prop}\label{p.centralfiber}
Assume $\mbA \neq \mbC$.
Consider a family of smooth projective varieties $\pi\colon \mcX \to \Delta$ with $\mcX_t \simeq X(\mbA)$ for all $t \neq 0$. Then
\begin{enumerate}
\item\label{p.cf-1} either $\mcX_0 \simeq X(\mbA)$,
\item\label{p.cf-2} or $\mathfrak{aut}(\mcX_0) \simeq \mathbb{C}^n\rtimes (\mathfrak{so}_3(\mbA) \oplus \mathbb{C})$ and $\mcX_0$ is an equivariant compactification of $\mathbb{G}_a^n$ with $n =\dim X(\mbA)$.
\end{enumerate}
\end{prop}

To prove Theorem~\ref{t.main}, it suffices to exclude  case~\eqref{p.cf-2}  in Proposition~\ref{p.centralfiber}. In the following, we will assume case~\eqref{p.cf-2}  to deduce a contradiction.

\subsection{Reduction to a family of surfaces} \label{n.VW-bdl-Liealg}

Let $\mcV =\pi_* T_{\mcX/\Delta}$, which is a vector bundle over $\Delta$ such that $\mcV_t \simeq \mathfrak{aut}(\mcX_t) \simeq \mathfrak{sl}_3(\mbA)$ for $t\neq 0$.  Let $\mcW \subset \mcV$ be the subbundle such that $\mcW_t \simeq \mathfrak{so}_3(\mbA)$ for $t\neq 0$, which is the Lie algebra of the stabilizer of $\tau(t) \in \mcX_t$. It follows that $\mcW_0 \simeq \mathfrak{so}_3(\mbA)$.  For each $t\in \Delta$, the fiber $\mcV_t$ is a completely reducible $\mathfrak{so}_3(\mbA)$-module, which is isomorphic to $\mathfrak{so}_3(\mbA)\oplus\SJ_3(\mbA)_0$. By a dimension check, $\mcV_0 \simeq \mathbb{C}^n\rtimes \mathfrak{so}_3(\mbA) \subset \mathfrak{aut}(\mcX_0)$. Our construction and argument here for $\mcV$ and $\mcW$ is an analogue of the proof of~\cite[Lemma 4.11]{FL}.

We fix a family of Cartan subalgebra $\mcH \subset \mcW$; \textit{i.e.} $\mcH_t$ is a Cartan subalgebra of $\mcW_t$ for all $t$. Consider $\tilde{\mcH} \subset \mcV$ defined by
$$
\tilde{\mcH}_t := \{ v \in \mcV_t | [v, \mcH_t]=0\}.
$$
It follows that $\tilde{\mcH}_t$ is a Cartan subalgebra of $\mcV_t $ for all $t \neq 0$, by Lemma~\ref{l.Cartan2}.  Note that $\rk(\tilde{\mcH}) = \rk(\mcH)+2$ (as $\mbA \neq \mbC$).

 For $t \in \Delta$, let $\bH_t = \exp(\mcH_t) \subset \Aut^0(\mcX_t)_{\tau(t)}$, which is a family of tori. Set $\tilde{\bH}_t  = \exp(\tilde{\mcH}_t) \subset \Aut^0(\mcX_t)$.

 \begin{prop} \label{p.reduction}
Assume $\mbA\neq\mbC$. Let $\mcY \subset \mcX$ be the connected component of the fixed locus $\mcX^\bH$ along the section $\tau(\Delta) \subset \mcX$.
\begin{enumerate}
\item\label{p.r-1} The map $\mcY \to \Delta$ is a smooth family of projective surfaces.
\item\label{p.r-2} For each $t\in\Delta$, $\mcY_t$ is the closure of $\tilde{\bH}_t\cdot \tau(t)$ in $\mcX_t$.
\item\label{p.r-3} When $t\neq 0$, $\mcY_t$ is isomorphic to $Y(\mbA)$, the blowup of\, $\mbP^2$ along three coordinate points.
\item\label{p.r-4} The inclusion $\mcY_0 \subset \mcX_0$ is a smooth equivariant compactification of\, $\mathbb{G}_a^2$.
\end{enumerate}
 \end{prop}

\begin{proof}
By Bia{\l}ynicki-Birula's theorem on torus actions \cite{B}, the map $\mcY\to\Delta$ is a smooth family of projective varieties. For each $t\in\Delta$, the representation of $\mcW_t\simeq \mathfrak{so}_3(\mbA)$ on $T_{\tau(t)}\mcX_t$ coincides with that on $\SJ_3(\mbA)_0$, and the subspace $T_{\tau(t)}\mcY_t$ is contained in the $\bH_t$-eigenspace of weight zero, which is of dimension $2$. It follows that $\dim\mcY_t\leq 2$. Since $\tilde{\mcH}_t\cdot\tau(t)\subset\mcY_t$, we have $\dim\mcY_t\geq\dim\tilde{\mcH}_t\cdot\tau(t)=\dim(\tilde{\mcH}_t/\mcH_t)=2$. Then $\mcY_t$ is the closure of $\tilde{\bH}_t\cdot \tau(t)$ in $\mcX_t$, and it is a projective surface. This proves~\eqref{p.r-1} and~\eqref{p.r-2}.

By Proposition~\ref{p.MaxTorus}, when $t\neq 0$, the projective surface $\mcY_t$ is isomorphic to the closure of $T_0\cdot o$ in $X(\mbA)$. Then~\eqref{p.r-3} follows from Lemma~\ref{l.ToricFiber}. By the structure of $\mcV_0$, $\tilde{\bH}_0$ is the semi-direct product of the torus $\bH_0$ and a vector group $\mathbb{G}_a^2$. Then~\eqref{p.r-4} follows from~\eqref{p.r-2}.
\end{proof}

It follows that the $\mcY_t$ are quasi-homogeneous for all $t\in \Delta$.  Denote by $\partial\mcY_t$ the boundary, \textit{i.e.} the complement of the open orbit. Let $\partial\mcY$ be the closure of $\cup_{t\neq 0}\partial\mcY_t$, and let $(\partial\mcY)_t$ be the fiber of $\partial\mcY$ over $t\in\Delta$.

\begin{lem}\label{l.bound-Y0}
We have $(\partial\mcY)_t=\partial\mcY_t$ as sets for each $t\in\Delta$.
\end{lem}

\begin{proof}
When $t\neq 0$, it is immediate from the construction that $(\partial\mcY)_t=\partial\mcY_t$.  The subvariety $(\partial\mcY)_t\subset\mcY_t$ is stable under the vector fields in $\tilde{\mcH}_t$ for $t \neq 0$. By continuity, this is also the case for $t=0$. Consequently, $(\partial\mcY)_0$ has no intersection with the open orbit $\tilde{\bH}_t\cdot\tau(0)$ on $\mcY_0$, implying $(\partial\mcY)_0\subset\partial\mcY_0$ as sets.

Since $(\partial\mcY)_t=\partial\mcY_t $ is the anticanonical divisor on $\mcY_t$ when $t\neq 0$, $-K_{\mcY_0}$ is given by the divisor $(\partial\mcY)_0$ (as a scheme-theoretic divisor, so each irreducible component has a multiplicity). As $\mcY_0$ is an equivariant compactification of a vector group, the support of its $\mbG_a^2$-stable anticanonical divisor is the whole boundary by \cite[Theorem 2.7]{HT}. It follows that $(\partial\mcY)_0=\partial\mcY_0$ as sets.
\end{proof}

In the following, we will construct an involution that acts well on $\mathcal{Y}/\Delta$.

\begin{lem}
There is a direct sum decomposition of vector bundles $\mcV=\mcW\oplus \mathcal{M}$ over $\Delta$ which is a direct sum decomposition of irreducible $\mathfrak{so}_3(\mathbb{A})$-modules for all $t \in \Delta$.  Moreover, $\mcW_t\cong\mathfrak{so}_3(\mbA)$ for all $t\in \Delta$, while $\mathcal{M}_t=\SJ_3(\mbA)_0$ for $t \neq 0$ and $\mathcal{M}_0=\mathbb{C}^n$ is the radical of the Lie algebra $ \mcV_0$, where $n=\dim X(\mathbb{A})$.
\end{lem}

\begin{proof}
For each $t\neq 0$, $\mcW_t\cong\mathfrak{so}_3(\mbA)$ is contained in the isotropic subalgebra at $\tau(t)\in\mcX_t$, and $\mcV_t/\mcW_t$ is an irreducible representation of $\mathfrak{so}_3(\mbA)$ isomorphic to the representation $T_{\tau(t)}\mcX_t$ of the isotropic subalgebra. For each $t\in\Delta$, the evaluation of vector fields at $\tau(t)\in\mcX_t$ gives rise to an injective homomorphism of $\mcW_t$-modules $\mcV_t/\mcW_t\to T_{\tau(t)}\mcX_t$, whence $\mcV_t/\mcW_t$ is isomorphic to the irreducible $\mathfrak{so}_3(\mbA)$-module $T_oX(\mbA)$ by Proposition~\ref{p.Invariance} and by the fact that $\dim\mcV_t/\mcW_t=\dim\mcX_t$. Since $\mathfrak{so}_3(\mbA)$ is a simple Lie algebra, for each $t\in\Delta$, the module $\mcV_t$ is isomorphic to $\mathfrak{so}_3(\mbA)\oplus T_oX(\mbA)$. Since the two direct summands $\mathfrak{so}_3(\mbA)$ and $T_oX(\mbA)$ are irreducible modules that are not isomorphic to each other, this decomposition is unique. Then we obtain a direct sum decomposition of the holomorphic family $\mcV/\Delta$ of $\mathfrak{so}_3(\mbA)$-modules $\mcV=\mcW\oplus\mcM$. When $t\neq 0$, the identification $\mcV_t=\mathfrak{sl}_3(\mbA)$ gives rise to the identification $\mcM_t=\SJ_3(\mbA)_0$. When $t=0$, $\mcV_0\simeq \mathbb{C}^n\rtimes \mathfrak{so}_3(\mbA)$ is already a decomposition into irreducible $\mathfrak{so}_3(\mbA)$-modules. By the uniqueness of the decomposition, $\mcM_0=\mbC^n$ is the radical of $\mcV_0$.
\end{proof}

For each $t\in\Delta$, define $\xi_t(\phi)=\phi$ for $\phi\in\mcW_t$ and $\xi_t(\phi)=-\phi$ for $\phi\in\mathcal{M}_t$. This gives an automorphism $\xi$ of the vector bundle $\mcV/\Delta$ of order $2$. When $t\neq 0$, $\xi_t$ is nothing but the involution $ \theta\in\Aut(\mathfrak{so}_3(\mbA))$.

\begin{prop}\label{p.involution-family}\leavevmode
    \begin{enumerate}
        \item\label{p.i-f-1} The element $\xi\in\Aut(\mcV/\Delta)$ induces an involution $\Theta$ of $\mathcal{X}/\Delta$.
        \item\label{p.i-f-2} We have $\tau(\Delta)\subset \mathcal{X}^\Theta$ and $\Theta_t=\theta$ for $t\neq 0$.
        \item\label{p.i-f-3} For each $t\in \Delta$, the induced map $(\Theta_t)_*\in\GL(T_{\tau(t)}\mathcal{X}_t)$ is just $-1$.
        \item\label{p.i-f-4} For each $t\in\Delta$, $\Theta_t(\mathcal{Y}_t)=\mathcal{Y}_t$ and $\Theta_t(\partial \mathcal{Y}_t)=\partial \mathcal{Y}_t$.
    \end{enumerate}
\end{prop}

\begin{proof}
\eqref{p.i-f-1}  Take any $t\in \Delta$. Let $G_t$ be the connected algebraic subgroup of $\Aut^0(\mcX_t)$ with Lie algebra $\mcV_t\subset\aut(\mcX_t)$, and denote by $H_t\subset G_t$ the isotropic subgroup at $\tau(t)\in\mcX_t$. Then $\mcX_t^o:=G_t\cdot\tau(t)\cong G_t/H_t$ is the open orbit of $\mcX_t$.
Recall that $\mcW_t=\{\phi\in\mcV_t\mid \xi(\phi)=\phi\}$.
Then $\xi$ induces a biholomorphic map $\Theta:\mathcal{X}^o:=\cup_{t\in \Delta}\mathcal{X}_t^o\rightarrow\mathcal{X}^o$ over $\Delta$, and $\Theta\circ\Theta=\id$.

When $t\neq 0$, $\Theta_t|_{\mathcal{X}_t^o}=\theta|_{\SL_3(\mathbb{A})/\SO_3(\mathbb{A})}$, and thus $d\Theta_t$ preserves the VMRT of $\mathcal{X}_t^o$. By continuity, $d\Theta_0$ preserves the VMRT of $\mathcal{X}_0^o$. By the extension theorem of Cartan--Fubini type \cite[Main Theorem]{HM01}, we can extend $\Theta_0$ to a biholomorphic map $\mcX_0 \to \mcX_0$.
When $t\neq 0$, we have identifications $\mathcal{X}_t^o=\SL_3(\mathbb{A})/\SO_3(\mathbb{A})$ and $\Theta_t=\theta$. Then~\eqref{p.i-f-2} and~\eqref{p.i-f-3} follow.

    \eqref{p.i-f-4} Note that $\mcH\subset\mcW$ and $\tilde{\mcH}\subset\mcW$ are $\xi$-stable. Then the open orbit of $\mcY_t$, $t\in\Delta$, is $\Theta_t$-stable. Hence the closure $\mcY_t$ and the boundary $\partial\mcY_t$ is $\Theta_t$-stable.
\end{proof}

\subsection{The central fiber as a blowup of $\mathbb{P}^2$}

Recall that general fibers of the smooth family $\mcY/\Delta$ of rational projective surfaces are of Picard number~$4$, as is the special fiber $\mcY_0$. As $\mcY_0$ is an equivariant compactification of a vector group, its boundary is of pure codimension $1$ and spans $\Pic(\mcY_0)$ freely. In particular, $\partial \mathcal{Y}_0$ has four irreducible components, say $F_0$, $F_1$, $F_2$ and $F_3$; thus $\Pic(\mcY_0)=\oplus_{i=0}^3\mbZ[F_i]$.

Denote by $\mcD_i$ (resp.\ $\mcE_j$) the prime divisor on $\mcY/\Delta$ such that $\mcD_{i, t}=D_i$ (resp.\ $\mcE_{j, t}=E_j$) under the identification $\mcY_t=Y(\mbA)$ for $t\neq 0$. By Lemma~\ref{l.bound-Y0}, the divisors $\mcD_{i, 0}$ and $\mcE_{j, 0}$ of $\mcY_0$ lie in the boundary. We will find out what $\mcD_{i, 0}$ and $\mcE_{j, 0}$ are in the following.

There is an $\SO_3(\mbA)$-action on the family $\mcX/\Delta$ that is isotropic along the section $\tau(\Delta)$, and the associated Lie algebra is $\mcW_t\cong\mathfrak{so}_3(\mbA)$ for each $t\in\Delta$.

\begin{lem}\label{l.S3-family Y}
For each $t\in\Delta$, the subvariety $\mcY_t\subset\mcX_t$ as well as the boundary $\partial\mcY_t$ are stable under the action of $\mathfrak{S}_3\subset\SO_3(\mbA)$. Furthermore, $\sigma\cdot\mcD_i=\mcD_{\sigma(i)}$ and $\sigma\cdot\mcE_j=\mcE_{\sigma(j)}$ for $\sigma\in\mathfrak{S}_3$ and $1\leq i, j\leq 3$.
\end{lem}

\begin{proof}
By Proposition~\ref{p.S3action}, $\mcY_t$ and $\partial\mcY_t$ are $\mathfrak{S}_3$-stable for $t \neq 0$.  By continuity, so are $\mcY_0$ and $\partial\mcY_0$. Since $\sigma\cdot\mcD_{i, t}=\mcD_{\sigma(i), t}$ and $\sigma\cdot\mcE_{j, t}=\mcE_{\sigma(j), t}$ for $t\neq 0$, we have  $\sigma\cdot\mcD_i=\mcD_{\sigma(i)}$ and $\sigma\cdot\mcE_j=\mcE_{\sigma(j)}$.
\end{proof}

\begin{prop} The four irreducible components of $\partial \mathcal{Y}_0$ can be denoted by $F_0$, $F_1$, $F_2$, $F_3$ such that $\sigma(F_0)=F_0$ and $\sigma(F_i)=F_{\sigma(i)},$ where $\sigma\in\mathfrak{S}_3$ and $i=1,2,3$.
\end{prop}

\begin{proof}
For any $t\in \Delta$,
the restriction map $\Pic(\mcY/\Delta)\to\Pic(\mcY_t)$ is an isomorphism which is compatible with the $\mathfrak{S}_3$-action. Then $\Pic(\mcY_0)^\Gamma\simeq\Pic(\mcY_t)^\Gamma$ for any subgroup $\Gamma$ of $\mathfrak{S}_3$ and any $t\neq 0$. In particular, $\Pic(\mathcal{Y}_0)^{\sigma_{12}}\simeq \Pic(\mathcal{Y}_t)^{\sigma_{12}}$ is of rank 3 by Proposition~\ref{p.Picard-invariant}.

 Let $\mcS:=\{F_0,F_1,F_2,F_3\}$ be the set of irreducible components of $\partial\mathcal{Y}_0$.  By \cite[Lemma 4.16]{FL}, the rank of $\Pic(\mathcal{Y}_0)^{\sigma_{12}}$ is given by the number of $\sigma_{12}$-orbits on $\mcS$. Hence $\mcS$ consists of three orbits under the $\sigma_{12}$-action; we may assume they are $\{F_1,F_2\}$, $\{F_3\}$ and $\{F_0\}$.

Applying the action of $\sigma_{23}$ on $\mathcal{Y}/\Delta$, by a similar argument there exists a subset $\mcT \subset\mcS$ such that $\mcT$ consists of two elements permuted by $\sigma_{23}$ and each element of the set $\mcS\setminus \mcT$ is  $\sigma_{23}$-stable.

We claim that $\mcT\cap\{F_1,F_2\}$ consists of a unique element. Otherwise, either $\mcT=\{F_1,F_2\}$ or $\mcT=\{F_0,F_3\}$. In both cases, the action of $\sigma_{123}=\sigma_{12}\circ\sigma_{23}$ on $\mcS$ is of order at most $2$, and thus $\sigma_{321}=\sigma_{123}\circ\sigma_{123}$ acts trivially on $\mcS$. It follows that $\Pic(\mathcal{Y}_0)^{\sigma_{321}}\simeq \Pic(\mathcal{Y}_0)$ is of rank~$4$. However, $\Pic(\mathcal{Y}_0)^{\sigma_{321}}\simeq\Pic(\mathcal{Y}_t)^{\sigma_{321}}$ is of rank $2$ by Proposition~\ref{p.Picard-invariant}, which gives a contradiction.

By the claim above, we may assume $\mcT = \{F_2, F_3\}$ up to re-ordering. It follows  that $\sigma(F_0)=F_0$ and $\sigma(F_i)=F_{\sigma(i)}$, $i=1,2,3$, for any $\sigma\in \mathfrak{S}_3$.
\end{proof}

\begin{lem} \label{l.Y0-D-E}
There are non-negative integers $d_0$, $d_1$, $d_2$ and $e_0$, $e_1$, $e_2$ such that for $\sigma\in \mathfrak{S}_3$,
\begin{eqnarray*}
\mcD_{1,0}=d_0F_0+d_1F_1+d_2F_2+d_2F_3, & \mcD_{\sigma(1),0}=d_0F_0+d_1F_{\sigma(1)}+d_2F_{\sigma(2)}+d_2F_{\sigma(3)}, \\
\mcE_{1,0}=e_0F_0+e_1F_1+e_2F_2+e_2F_3, & \mcE_{\sigma(1),0}=e_0F_0+e_1F_{\sigma(1)}+e_2F_{\sigma(2)}+e_2F_{\sigma(3)}.
\end{eqnarray*}
\end{lem}

\begin{proof}
Since $\mcD_{i, t}$ and $\mcE_{i, t}$ are contained in $\partial\mcY_t$ when $t\neq 0$, the same holds when $t=0$. Then there exist non-negative integers $d_0$, $d_1$, $d_2$, $d_3$ such that $\mcD_{1, 0}=d_0F_0+d_1F_1+d_2F_2+d_3F_3.$ For each $\sigma\in \mathfrak{S}_3$, we have $\mcD_{\sigma(1),0}=d_0F_0+d_1F_{\sigma(1)}+d_2F_{\sigma(2)}+d_3F_{\sigma(3)}$.
Applying the formula to $\sigma=\sigma_{23}$, we have $\mcD_{1, 0}=d_0F_0+d_1F_1+d_2F_3+d_3F_2$, implying $d_2=d_3$. The conclusions for $\mcE_{1, 0}$ and $\mcE_{\sigma(1), 0}$ can be obtained similarly.
\end{proof}

Now we compute the anticanonical divisor of the central fiber $\mcY_0$.

\begin{cor} \label{c.-K_Y0}
We have $-K_{\mathcal{Y}_0}=3(d_0+e_0)F_0+(d_1+2d_2+e_1+2e_2)(\sum_{i=1}^3F_i)$. Moreover, $d_0+e_0\geq 1$ and $d_1+e_2=d_2+e_1\geq 1$.
\end{cor}

\begin{proof}
Since $-K_{Y(\mbA)}=\sum_{i=1}^3(D_i+E_i)$, we have $-K_{\mcY/\Delta}=\sum_{i=1}^3(\mcD_i+\mcE_i)$. It follows that $-K_{\mathcal{Y}_0}=3(d_0+e_0)F_0+(d_1+2d_2+e_1+2e_2)\sum_{i=1}^3F_i$.

When $t\neq 0$, we have $\mcD_{1, t}-\mcD_{2, t}=\mcE_{1, t}-\mcE_{2, t}\in\Pic(\mcY_t)=\Pic(Y(\mbA))$ by Lemma~\ref{l.D-E-lin-equivalence}. Then the same holds for $t=0$ by applying the identifications $\Pic(\mcY_0)=\Pic(\mcY/\Delta)=\Pic(Y(\mbA))$. On the other hand, by Lemma~\ref{l.Y0-D-E}, we have $\mcD_{1, 0}-\mcD_{\sigma_{12}(1), 0}=(d_1-d_2)(F_1-F_2)$ and $\mcE_{1, 0}-\mcE_{\sigma_{12}(1), 0}=(e_1-e_2)(F_1-F_2)$. It follows that $d_1-d_2=e_1-e_2$, \textit{i.e.} $d_1+e_2=d_2+e_1$.

By \cite[Theorem 2.7]{HT}, the support of $-K_{\mcY_0}$ is the whole boundary, which implies $d_0+e_0\geq 1$ and $d_1+e_2=d_2+e_1\geq 1$.
\end{proof}

To prove that $\mathcal{Y}_0$ is the blowup of $\mathbb{P}^2$ along three colinear points, we start with the following.

\begin{prop}[\textit{cf.} \protect{\cite[Section~5]{HT}}] \label{p.equivariant-min-surf}
Every $\mathbb{G}_a^2$-surface admits a $\mathbb{G}_a^2$-equivariant morphism onto $\mathbb{P}^2$ or a Hirzebruch surface $\mathbb{F}_n$. The boundary of\, $\mbP^2$ consists of a unique line. The boundary of\, $\mbF_n$ consists of two lines; one is a fiber, and the other is a minimal section.
\end{prop}

\begin{prop}
The central fiber $\mathcal{Y}_0$ is a $\mathbb{G}_a^2$-equivariant blowup of\, $\mathbb{P}^2$, $\mathbb{F}_0$ or $\mathbb{F}_1$.
\end{prop}

\begin{proof}
Suppose there is a $\mathbb{G}_a^2$-equivariant blowdown $\mathcal{Y}_0\rightarrow \mathbb{F}_n$ with $n\geq 2$.  Let $l_1$ and $l_2$ be two lines of $\mathbb{F}_n$ such that $\mathbb{F}_n\setminus \mathbb{C}^2=l_1\cup l_2$, where $l_1$ is the section of $\mathbb{F}_n\rightarrow \mathbb{P}^1$ and $l_2$ is a fiber of $\mathbb{F}_n\rightarrow \mathbb{P}^1$. Then the anticanonical divisor of $\mathbb{F}_n$ is given by $-K_{\mathbb{F}_n}=2l_2+(n+2)l_1$.

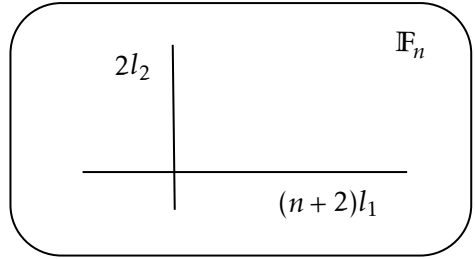
\begin{center}
\tikzset{every picture/.style={line width=0.75pt}} 
\begin{tikzpicture}[x=0.75pt,y=0.75pt,yscale=-1,xscale=1]

\draw   (216,947.25) .. controls (216,933.22) and (227.37,921.85) .. (241.4,921.85) -- (420.6,921.85) .. controls (434.63,921.85) and (446,933.22) .. (446,947.25) -- (446,1023.45) .. controls (446,1037.48) and (434.63,1048.85) .. (420.6,1048.85) -- (241.4,1048.85) .. controls (227.37,1048.85) and (216,1037.48) .. (216,1023.45) -- cycle ;
\draw    (252,1006.85) -- (414,1006.85) ;
\draw    (297,942.85) -- (298,1025.85) ;

\draw (375,1023) node  [color={rgb, 255:red, 0; green, 0; blue, 0 }  ,opacity=1 ,rotate=-358.71]  {$( n+2) l_{1}$};
\draw (277,954) node  [color={rgb, 255:red, 0; green, 0; blue, 0 }  ,opacity=1 ,rotate=-358.71]  {$2l_{2}$};
\draw (416,943) node  [color={rgb, 255:red, 0; green, 0; blue, 0 }  ,opacity=1 ,rotate=-1.12]  {$\mathbb{F}_{n}$};
\end{tikzpicture}
\end{center}

Since $n\geq 2$, the blowup of $\mathbb{F}_n$ along any point on $l_1\cup l_2$ would yield a surface $S$ and an exceptional divisor $E$ such that $-K_S=al_1+bl_2+cE$, where $a,b,c$ are distinct positive integers.

Any further blowup of $S$ will produce a surface $\tilde{S}$ whose anticanonical divisor $-K_{\tilde{S}}$  has at least three distinct coefficients, which is different from the form
 $-K_{\mathcal{Y}_0}=aF_0+b(F_1+F_2+F_3)$.
Hence by Proposition~\ref{p.equivariant-min-surf}, $\mathcal{Y}_0$ is the $\mathbb{G}_a^2$-equivariant blowup of $\mathbb{P}^2$, $\mathbb{F}_0$ or $\mathbb{F}_1$. \end{proof}

\begin{prop} \label{p.Y0-DE}
  There is a $\mathbb{G}_a^2$-equivariant birational morphism $\Phi\colon \mathcal{Y}_0\rightarrow \mathbb{P}^2$ such that $l_0=\Phi(F_0)$ is the boundary $\partial \mathbb{P}^2=\mathbb{P}^2\setminus \mathbb{C}^2$ and $p_i=\Phi(F_i)$, $1\leq i\leq 3$, are three distinct points on $l_0$.
Moreover, $-K_{\mathcal{Y}_0}=3F_0+2(F_1+F_2+F_3)$, and we have
\begin{enumerate}
\item\label{p.YO-DE-1} either $\mcD_{i,0}=F_0+F_i$, $1\leq i\leq 3$, and $\mcE_{j, 0}=F_j$, $1\leq j\leq 3$,
\item\label{p.YO-DE-2} or $\mcD_{i,0}=F_i$, $1\leq i\leq 3$, and $\mcE_{j,0}=F_0+F_j$, $1\leq j\leq 3$.
\end{enumerate}
\end{prop}

\begin{proof}
In the following, we will apply the similar idea of comparing coefficients of $-K_{\mathcal{Y}_0}=aF_0+b(F_1+F_2+F_3)$ and $\mathbb{G}_a^2$-equivariant blowups of $\mathbb{P}^2$, $\mathbb{F}_0=\mathbb{P}^1\times \mathbb{P}^1$ and $\mathbb{F}_1$.

{\it Case 1: blow up from $\mathbb{P}^2$.} There is only one possibility of blowing up $\mathbb{P}^2$ to get $\mathcal{Y}_0$, that is, blow up three distinct points on $l_0=\mathbb{P}^2\setminus{\mathbb{C}^2}$.

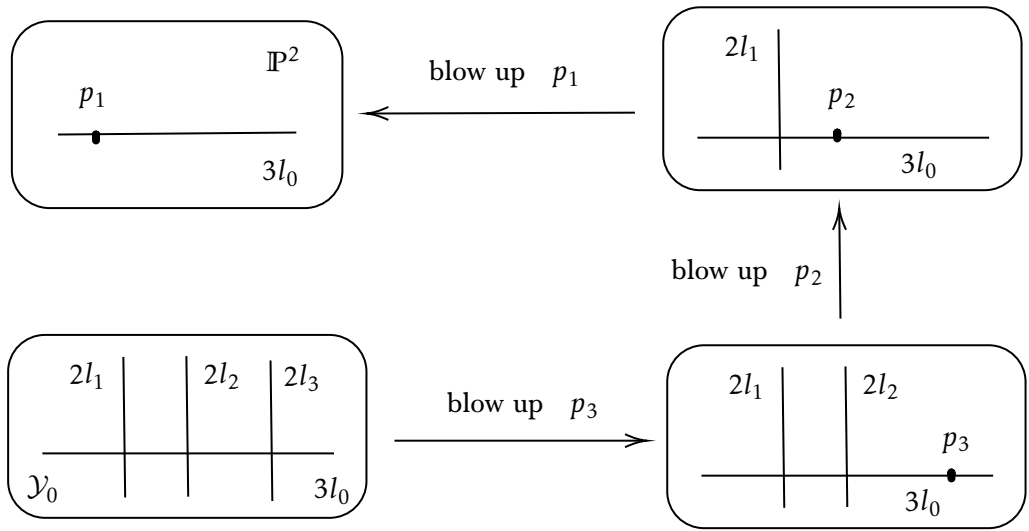
\begin{center}
\tikzset{every picture/.style={line width=0.75pt}} 

\begin{tikzpicture}[x=0.75pt,y=0.75pt,yscale=-1,xscale=1]

\draw   (330.14,941.25) .. controls (330.14,931.09) and (338.38,922.85) .. (348.54,922.85) -- (490.6,922.85) .. controls (500.76,922.85) and (509,931.09) .. (509,941.25) -- (509,996.45) .. controls (509,1006.61) and (500.76,1014.85) .. (490.6,1014.85) -- (348.54,1014.85) .. controls (338.38,1014.85) and (330.14,1006.61) .. (330.14,996.45) -- cycle ;
\draw    (347.22,988.44) -- (492.82,988.44) ;
\draw    (387.66,933.92) -- (388.56,1004.63) ;
\draw   (5,948.65) .. controls (5,938.27) and (13.42,929.85) .. (23.8,929.85) -- (151.2,929.85) .. controls (161.58,929.85) and (170,938.27) .. (170,948.65) -- (170,1005.05) .. controls (170,1015.43) and (161.58,1023.85) .. (151.2,1023.85) -- (23.8,1023.85) .. controls (13.42,1023.85) and (5,1015.43) .. (5,1005.05) -- cycle ;
\draw    (28,986.85) -- (147,985.85) ;
\draw    (316,975.85) -- (184,976.29) ;
\draw [shift={(182,976.29)}, rotate = 359.81] [color={rgb, 255:red, 0; green, 0; blue, 0 }  ][line width=0.75]    (10.93,-3.29) .. controls (6.95,-1.4) and (3.31,-0.3) .. (0,0) .. controls (3.31,0.3) and (6.95,1.4) .. (10.93,3.29)   ;
\draw  [line width=3] [line join = round][line cap = round] (47,989.85) .. controls (47,987.52) and (47,985.52) .. (47,987.85) ;
\draw    (418.61,1079.44) -- (418.02,1025.85) ;
\draw [shift={(418,1023.85)}, rotate = 89.37] [color={rgb, 255:red, 0; green, 0; blue, 0 }  ][line width=0.75]    (10.93,-3.29) .. controls (6.95,-1.4) and (3.31,-0.3) .. (0,0) .. controls (3.31,0.3) and (6.95,1.4) .. (10.93,3.29)   ;
\draw   (332.14,1111.25) .. controls (332.14,1101.09) and (340.38,1092.85) .. (350.54,1092.85) -- (492.6,1092.85) .. controls (502.76,1092.85) and (511,1101.09) .. (511,1111.25) -- (511,1166.45) .. controls (511,1176.61) and (502.76,1184.85) .. (492.6,1184.85) -- (350.54,1184.85) .. controls (340.38,1184.85) and (332.14,1176.61) .. (332.14,1166.45) -- cycle ;
\draw    (349.22,1158.44) -- (494.82,1158.44) ;
\draw    (389.66,1103.92) -- (390.56,1174.63) ;
\draw  [line width=3] [line join = round][line cap = round] (417,988.85) .. controls (417,986.52) and (417,984.52) .. (417,986.85) ;
\draw    (421.12,1103.5) -- (422.02,1174.2) ;
\draw   (3.14,1106.05) .. controls (3.14,1095.89) and (11.38,1087.65) .. (21.54,1087.65) -- (163.6,1087.65) .. controls (173.76,1087.65) and (182,1095.89) .. (182,1106.05) -- (182,1161.25) .. controls (182,1171.41) and (173.76,1179.65) .. (163.6,1179.65) -- (21.54,1179.65) .. controls (11.38,1179.65) and (3.14,1171.41) .. (3.14,1161.25) -- cycle ;
\draw    (20.22,1147.24) -- (165.82,1147.24) ;
\draw    (60.66,1098.72) -- (61.56,1169.43) ;
\draw    (92.12,1098.3) -- (93.02,1169) ;
\draw    (134.12,1100.5) -- (135.02,1171.2) ;
\draw    (196,1140.85) -- (320,1140.85) ;
\draw [shift={(322,1140.85)}, rotate = 180] [color={rgb, 255:red, 0; green, 0; blue, 0 }  ][line width=0.75]    (10.93,-3.29) .. controls (6.95,-1.4) and (3.31,-0.3) .. (0,0) .. controls (3.31,0.3) and (6.95,1.4) .. (10.93,3.29)   ;
\draw  [line width=3] [line join = round][line cap = round] (474,1159.85) .. controls (474,1157.52) and (474,1155.52) .. (474,1157.85) ;

\draw (457.77,1002.2) node  [color={rgb, 255:red, 0; green, 0; blue, 0 }  ,opacity=1 ,rotate=-358.71]  {$3l_{0}$};
\draw (369.69,943.42) node  [color={rgb, 255:red, 0; green, 0; blue, 0 }  ,opacity=1 ,rotate=-358.71]  {$2l_{1}$};
\draw (139,1006) node  [color={rgb, 255:red, 0; green, 0; blue, 0 }  ,opacity=1 ,rotate=-358.71]  {$3l_{0}$};
\draw (140.62,949.05) node  [color={rgb, 255:red, 0; green, 0; blue, 0 }  ,opacity=1 ,rotate=-1.12]  {$\mathbb{P}^{2}$};
\draw (45.69,968.42) node  [color={rgb, 255:red, 0; green, 0; blue, 0 }  ,opacity=1 ,rotate=-358.71]  {$p_{1}$};
\draw (459.77,1172.2) node  [color={rgb, 255:red, 0; green, 0; blue, 0 }  ,opacity=1 ,rotate=-358.71]  {$3l_{0}$};
\draw (371.69,1113.42) node  [color={rgb, 255:red, 0; green, 0; blue, 0 }  ,opacity=1 ,rotate=-358.71]  {$2l_{1}$};
\draw (419.57,968.85) node  [color={rgb, 255:red, 0; green, 0; blue, 0 }  ,opacity=1 ,rotate=-358.71]  {$p_{2}$};
\draw (438.69,1113.42) node  [color={rgb, 255:red, 0; green, 0; blue, 0 }  ,opacity=1 ,rotate=-358.71]  {$2l_{2}$};
\draw (282,958) node  [color={rgb, 255:red, 0; green, 0; blue, 0 }  ,opacity=1 ,rotate=-358.71]  {$p_{1}$};
\draw (212,947.85) node [anchor=north west][inner sep=0.75pt]   [align=left] {blow up};
\draw (403,1058) node  [color={rgb, 255:red, 0; green, 0; blue, 0 }  ,opacity=1 ,rotate=-358.71]  {$p_{2}$};
\draw (333,1047.85) node [anchor=north west][inner sep=0.75pt]   [align=left] {blow up};
\draw (164.77,1166) node  [color={rgb, 255:red, 0; green, 0; blue, 0 }  ,opacity=1 ,rotate=-358.71]  {$3l_{0}$};
\draw (42.69,1108.22) node  [color={rgb, 255:red, 0; green, 0; blue, 0 }  ,opacity=1 ,rotate=-358.71]  {$2l_{1}$};
\draw (109.69,1108.22) node  [color={rgb, 255:red, 0; green, 0; blue, 0 }  ,opacity=1 ,rotate=-358.71]  {$2l_{2}$};
\draw (149.69,1109.22) node  [color={rgb, 255:red, 0; green, 0; blue, 0 }  ,opacity=1 ,rotate=-358.71]  {$2l_{3}$};
\draw (291,1125) node  [color={rgb, 255:red, 0; green, 0; blue, 0 }  ,opacity=1 ,rotate=-358.71]  {$p_{3}$};
\draw (221,1114.85) node [anchor=north west][inner sep=0.75pt]   [align=left] {blow up};
\draw (476.57,1141.85) node  [color={rgb, 255:red, 0; green, 0; blue, 0 }  ,opacity=1 ,rotate=-358.71]  {$p_{3}$};
\draw (20.62,1165.05) node  [color={rgb, 255:red, 0; green, 0; blue, 0 }  ,opacity=1 ,rotate=-1.12]  {$\mathcal{Y}_0$};

\end{tikzpicture}
\end{center}

{\it Case 2: blow up from $\mathbb{P}^1\times \mathbb{P}^1$.} There is only one possibility of blowing up  $\mathbb{F}_0=\mathbb{P}^1\times \mathbb{P}^1$ to get $\mathcal{Y}_0$, as follows, where $l_1\cup l_2=\mathbb{F}_0\setminus \mathbb{C}^2$, $l_3$ is the exceptional divisor of the first blowup of the point $p_1\in l_1\cap l_2$, and  $l_4$ is the exceptional divisor of the second blowup of a point $p_2\in l_3\setminus (l_1\cup l_2)$.

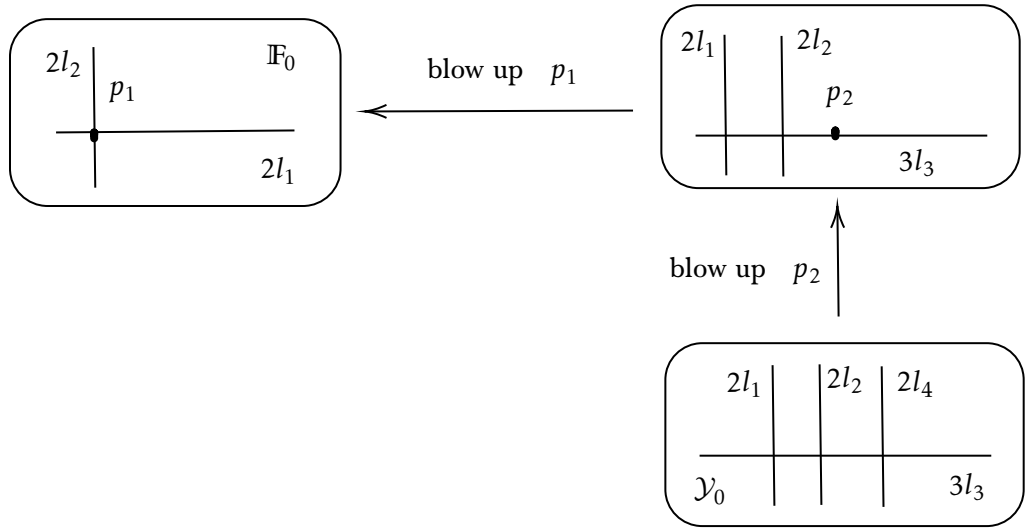
\begin{center}

\tikzset{every picture/.style={line width=0.75pt}} 

\begin{tikzpicture}[x=0.75pt,y=0.75pt,yscale=-1,xscale=1]

\draw   (330.14,941.25) .. controls (330.14,931.09) and (338.38,922.85) .. (348.54,922.85) -- (490.6,922.85) .. controls (500.76,922.85) and (509,931.09) .. (509,941.25) -- (509,996.45) .. controls (509,1006.61) and (500.76,1014.85) .. (490.6,1014.85) -- (348.54,1014.85) .. controls (338.38,1014.85) and (330.14,1006.61) .. (330.14,996.45) -- cycle ;
\draw    (347.22,988.44) -- (492.82,988.44) ;
\draw    (46.66,943.92) -- (47.56,1014.63) ;
\draw   (5,948.65) .. controls (5,938.27) and (13.42,929.85) .. (23.8,929.85) -- (151.2,929.85) .. controls (161.58,929.85) and (170,938.27) .. (170,948.65) -- (170,1005.05) .. controls (170,1015.43) and (161.58,1023.85) .. (151.2,1023.85) -- (23.8,1023.85) .. controls (13.42,1023.85) and (5,1015.43) .. (5,1005.05) -- cycle ;
\draw    (28,986.85) -- (147,985.85) ;
\draw    (316,975.85) -- (184,976.29) ;
\draw [shift={(182,976.29)}, rotate = 359.81] [color={rgb, 255:red, 0; green, 0; blue, 0 }  ][line width=0.75]    (10.93,-3.29) .. controls (6.95,-1.4) and (3.31,-0.3) .. (0,0) .. controls (3.31,0.3) and (6.95,1.4) .. (10.93,3.29)   ;
\draw  [line width=3] [line join = round][line cap = round] (47,989.85) .. controls (47,987.52) and (47,985.52) .. (47,987.85) ;
\draw    (418.61,1079.44) -- (418.02,1025.85) ;
\draw [shift={(418,1023.85)}, rotate = 89.37] [color={rgb, 255:red, 0; green, 0; blue, 0 }  ][line width=0.75]    (10.93,-3.29) .. controls (6.95,-1.4) and (3.31,-0.3) .. (0,0) .. controls (3.31,0.3) and (6.95,1.4) .. (10.93,3.29)   ;
\draw   (332.14,1111.25) .. controls (332.14,1101.09) and (340.38,1092.85) .. (350.54,1092.85) -- (492.6,1092.85) .. controls (502.76,1092.85) and (511,1101.09) .. (511,1111.25) -- (511,1166.45) .. controls (511,1176.61) and (502.76,1184.85) .. (492.6,1184.85) -- (350.54,1184.85) .. controls (340.38,1184.85) and (332.14,1176.61) .. (332.14,1166.45) -- cycle ;
\draw    (349.22,1149.44) -- (494.82,1149.44) ;
\draw    (385.66,1103.92) -- (386.56,1174.63) ;
\draw  [line width=3] [line join = round][line cap = round] (417,988.85) .. controls (417,986.52) and (417,984.52) .. (417,986.85) ;
\draw    (409.12,1103.5) -- (410.02,1174.2) ;
\draw    (361.66,937.72) -- (362.56,1008.43) ;
\draw    (390.12,938.3) -- (391.02,1009) ;
\draw    (440.12,1104.5) -- (441.02,1175.2) ;

\draw (457.77,1002.2) node  [color={rgb, 255:red, 0; green, 0; blue, 0 }  ,opacity=1 ,rotate=-358.71]  {$3l_{3}$};
\draw (31.69,952.42) node  [color={rgb, 255:red, 0; green, 0; blue, 0 }  ,opacity=1 ,rotate=-358.71]  {$2l_{2}$};
\draw (139,1006) node  [color={rgb, 255:red, 0; green, 0; blue, 0 }  ,opacity=1 ,rotate=-358.71]  {$2l_{1}$};
\draw (140.62,949.05) node  [color={rgb, 255:red, 0; green, 0; blue, 0 }  ,opacity=1 ,rotate=-1.12]  {$\mathbb{F}_0$};
\draw (61.69,966.42) node  [color={rgb, 255:red, 0; green, 0; blue, 0 }  ,opacity=1 ,rotate=-358.71]  {$p_{1}$};
\draw (482.77,1165.2) node  [color={rgb, 255:red, 0; green, 0; blue, 0 }  ,opacity=1 ,rotate=-358.71]  {$3l_{3}$};
\draw (371.69,1113.42) node  [color={rgb, 255:red, 0; green, 0; blue, 0 }  ,opacity=1 ,rotate=-358.71]  {$2l_{1}$};
\draw (419.57,968.85) node  [color={rgb, 255:red, 0; green, 0; blue, 0 }  ,opacity=1 ,rotate=-358.71]  {$p_{2}$};
\draw (421.69,1111.42) node  [color={rgb, 255:red, 0; green, 0; blue, 0 }  ,opacity=1 ,rotate=-358.71]  {$2l_{2}$};
\draw (282,958) node  [color={rgb, 255:red, 0; green, 0; blue, 0 }  ,opacity=1 ,rotate=-358.71]  {$p_{1}$};
\draw (212,947.85) node [anchor=north west][inner sep=0.75pt]   [align=left] {blow up};
\draw (403,1058) node  [color={rgb, 255:red, 0; green, 0; blue, 0 }  ,opacity=1 ,rotate=-358.71]  {$p_{2}$};
\draw (333,1047.85) node [anchor=north west][inner sep=0.75pt]   [align=left] {blow up};
\draw (348.69,942.22) node  [color={rgb, 255:red, 0; green, 0; blue, 0 }  ,opacity=1 ,rotate=-358.71]  {$2l_{1}$};
\draw (406.69,941.22) node  [color={rgb, 255:red, 0; green, 0; blue, 0 }  ,opacity=1 ,rotate=-358.71]  {$2l_{2}$};
\draw (456.69,1113.22) node  [color={rgb, 255:red, 0; green, 0; blue, 0 }  ,opacity=1 ,rotate=-358.71]  {$2l_{4}$};
\draw (354.62,1167.05) node  [color={rgb, 255:red, 0; green, 0; blue, 0 }  ,opacity=1 ,rotate=-1.12]  {$\mathcal{Y}_{0}$};

\end{tikzpicture}

\end{center}

{\it Case 3: blow up from $\mathbb{F}_1$.} There is only one possibility of blowing up $\mathbb{F}_1$ to get $\mathcal{Y}_0$, as follows, where $l_1\cup l_2=\mathbb{F}_1\setminus \mathbb{C}^2$ such that $-K_{\mathbb{F}_1}=3l_1+ 2l_2$, $l_3$ is the exceptional divisor of the first blowup of a point $p_1\in l_1\setminus l_2$, and $l_4$ is the exceptional divisor of the second blowup of a point $p_2\in l_1\setminus(l_2\cup l_3)$.

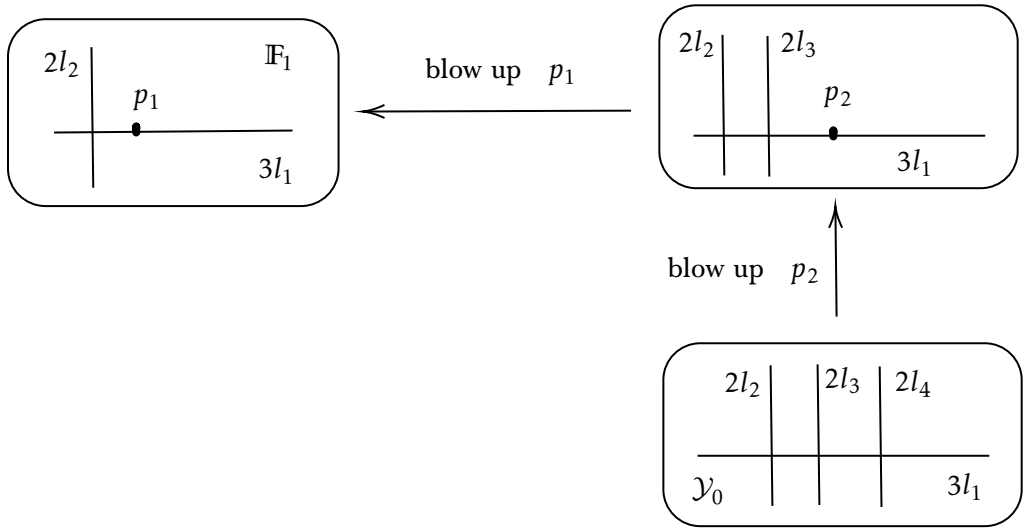
\begin{center}

\tikzset{every picture/.style={line width=0.75pt}} 

\begin{tikzpicture}[x=0.75pt,y=0.75pt,yscale=-1,xscale=1]

\draw   (330.14,941.25) .. controls (330.14,931.09) and (338.38,922.85) .. (348.54,922.85) -- (490.6,922.85) .. controls (500.76,922.85) and (509,931.09) .. (509,941.25) -- (509,996.45) .. controls (509,1006.61) and (500.76,1014.85) .. (490.6,1014.85) -- (348.54,1014.85) .. controls (338.38,1014.85) and (330.14,1006.61) .. (330.14,996.45) -- cycle ;
\draw    (347.22,988.44) -- (492.82,988.44) ;
\draw    (46.66,943.92) -- (47.56,1014.63) ;
\draw   (5,948.65) .. controls (5,938.27) and (13.42,929.85) .. (23.8,929.85) -- (151.2,929.85) .. controls (161.58,929.85) and (170,938.27) .. (170,948.65) -- (170,1005.05) .. controls (170,1015.43) and (161.58,1023.85) .. (151.2,1023.85) -- (23.8,1023.85) .. controls (13.42,1023.85) and (5,1015.43) .. (5,1005.05) -- cycle ;
\draw    (28,986.85) -- (147,985.85) ;
\draw    (316,975.85) -- (184,976.29) ;
\draw [shift={(182,976.29)}, rotate = 359.81] [color={rgb, 255:red, 0; green, 0; blue, 0 }  ][line width=0.75]    (10.93,-3.29) .. controls (6.95,-1.4) and (3.31,-0.3) .. (0,0) .. controls (3.31,0.3) and (6.95,1.4) .. (10.93,3.29)   ;
\draw  [line width=3] [line join = round][line cap = round] (69,986.85) .. controls (69,984.52) and (69,982.52) .. (69,984.85) ;
\draw    (418.61,1079.44) -- (418.02,1025.85) ;
\draw [shift={(418,1023.85)}, rotate = 89.37] [color={rgb, 255:red, 0; green, 0; blue, 0 }  ][line width=0.75]    (10.93,-3.29) .. controls (6.95,-1.4) and (3.31,-0.3) .. (0,0) .. controls (3.31,0.3) and (6.95,1.4) .. (10.93,3.29)   ;
\draw   (332.14,1111.25) .. controls (332.14,1101.09) and (340.38,1092.85) .. (350.54,1092.85) -- (492.6,1092.85) .. controls (502.76,1092.85) and (511,1101.09) .. (511,1111.25) -- (511,1166.45) .. controls (511,1176.61) and (502.76,1184.85) .. (492.6,1184.85) -- (350.54,1184.85) .. controls (340.38,1184.85) and (332.14,1176.61) .. (332.14,1166.45) -- cycle ;
\draw    (349.22,1149.44) -- (494.82,1149.44) ;
\draw    (385.66,1103.92) -- (386.56,1174.63) ;
\draw  [line width=3] [line join = round][line cap = round] (417,988.85) .. controls (417,986.52) and (417,984.52) .. (417,986.85) ;
\draw    (409.12,1103.5) -- (410.02,1174.2) ;
\draw    (361.66,937.72) -- (362.56,1008.43) ;
\draw    (384.12,938.3) -- (385.02,1009) ;
\draw    (440.12,1104.5) -- (441.02,1175.2) ;

\draw (457.77,1002.2) node  [color={rgb, 255:red, 0; green, 0; blue, 0 }  ,opacity=1 ,rotate=-358.71]  {$3l_{1}$};
\draw (31.69,952.42) node  [color={rgb, 255:red, 0; green, 0; blue, 0 }  ,opacity=1 ,rotate=-358.71]  {$2l_{2}$};
\draw (139,1006) node  [color={rgb, 255:red, 0; green, 0; blue, 0 }  ,opacity=1 ,rotate=-358.71]  {$3l_{1}$};
\draw (140.62,949.05) node  [color={rgb, 255:red, 0; green, 0; blue, 0 }  ,opacity=1 ,rotate=-1.12]  {$\mathbb{F}_{1}$};
\draw (74.69,970.42) node  [color={rgb, 255:red, 0; green, 0; blue, 0 }  ,opacity=1 ,rotate=-358.71]  {$p_{1}$};
\draw (482.77,1165.2) node  [color={rgb, 255:red, 0; green, 0; blue, 0 }  ,opacity=1 ,rotate=-358.71]  {$3l_{1}$};
\draw (371.69,1113.42) node  [color={rgb, 255:red, 0; green, 0; blue, 0 }  ,opacity=1 ,rotate=-358.71]  {$2l_{2}$};
\draw (419.57,968.85) node  [color={rgb, 255:red, 0; green, 0; blue, 0 }  ,opacity=1 ,rotate=-358.71]  {$p_{2}$};
\draw (421.69,1111.42) node  [color={rgb, 255:red, 0; green, 0; blue, 0 }  ,opacity=1 ,rotate=-358.71]  {$2l_{3}$};
\draw (282,958) node  [color={rgb, 255:red, 0; green, 0; blue, 0 }  ,opacity=1 ,rotate=-358.71]  {$p_{1}$};
\draw (212,947.85) node [anchor=north west][inner sep=0.75pt]   [align=left] {blow up};
\draw (403,1058) node  [color={rgb, 255:red, 0; green, 0; blue, 0 }  ,opacity=1 ,rotate=-358.71]  {$p_{2}$};
\draw (333,1047.85) node [anchor=north west][inner sep=0.75pt]   [align=left] {blow up};
\draw (348.69,942.22) node  [color={rgb, 255:red, 0; green, 0; blue, 0 }  ,opacity=1 ,rotate=-358.71]  {$2l_{2}$};
\draw (399.69,942.22) node  [color={rgb, 255:red, 0; green, 0; blue, 0 }  ,opacity=1 ,rotate=-358.71]  {$2l_{3}$};
\draw (456.69,1113.22) node  [color={rgb, 255:red, 0; green, 0; blue, 0 }  ,opacity=1 ,rotate=-358.71]  {$2l_{4}$};
\draw (354.62,1167.05) node  [color={rgb, 255:red, 0; green, 0; blue, 0 }  ,opacity=1 ,rotate=-1.12]  {$\mathcal{Y}_{0}$};

\end{tikzpicture}

\end{center}

All  three cases above yield the same $\mathcal{Y}_0$, which is the blowup of $\mathbb{P}^2$ along three colinear points.

Hence, $-K_{\mathcal{Y}_0}=3F_0+2(F_1+F_2+F_3)$. By Corollary~\ref{c.-K_Y0}, we have
$$d_0+e_0=d_1+e_2=d_2+e_1=1 \quad\text{and}\quad d_2+e_2=0$$
Indeed, by comparing the coefficients of $$-K_{\mathcal{Y}_0}=3F_0+2(F_1+F_2+F_3)=3(d_0+e_0)F_0+(d_1+2d_2+e_1+2e_2)(F_1+F_2+F_3)$$
we get $d_0+e_0=1$ and $d_1+2d_2+e_1+2e_2=(d_1+e_2)+(d_2+e_1)+(d_2+e_2)=2$. By Corollary~\ref{c.-K_Y0}, $d_i,e_j$, $0\leq i,j\leq 2$, are all non-negative and $d_1+e_2=d_2+e_1\geq 1$. So we have
$d_0+e_0=d_1+e_2=d_2+e_1=1$ and $d_2+e_2=0$. Thus
\begin{equation*}\pushQED{\qed}
  \text{either }\left\{\begin{array}{l}
     d_0=1  \\
     d_1=1\\
     d_2=0\\
     e_0=0\\
     e_1=1\\
     e_2=0
\end{array}\right.
\text{ or }\left\{\begin{array}{l}
     d_0=0  \\
     d_1=1\\
     d_2=0\\
     e_0=1\\
     e_1=1\\
     e_2=0.
\end{array}\right.
\end{equation*}
\qedhere
\end{proof}

\begin{rmk}
By choosing a family of three points in general position on $\mathbb{P}^2$ degenerating to three colinear points, we can construct a smooth projective family $\mathcal{Z}/\Delta$ such that  $\mathcal{Z}_t \simeq Y(\mbA)$ for each $t \neq 0$ while $\mathcal{Z}_0$ is a blowup of $\mathbb{P}^2$ along three colinear points. But in our situation, we have an extra involution $\Theta$ which prevents this situation.
\end{rmk}

Now we are ready to prove Theorem~\ref{t.main}.

\begin{proof}[Proof of Theorem~\ref{t.main}]
The case of $\mbA=\mbC$ follows from the classification of Mukai varieties.  Now assume $\mbA \neq \mbC$.

Take $\mcX \to \Delta$ to be a specialization of $X(\mbA)$. Assume that $\mcX_0$ is not isomorphic to $X(\mbA)$.  By Proposition~\ref{p.centralfiber}, $\mcX_0$ is an equivariant compactification of $\mathbb{G}_a^n$.  By Proposition~\ref{p.reduction}, we have a smooth family of surfaces $\mcY \to \Delta$ with the central fiber $\mcY_0$ being a $\mathbb{G}_a^2$-surface.  By Proposition~\ref{p.Y0-DE}, we may assume $\mcD_{i,0}=F_0+F_i$, $1\leq i\leq 3$, and $\mcE_{j,0}=F_j$, $1\leq j\leq 3$ (the proof for the other case is similar). By Lemma~\ref{p.involution-Y(A)} and Proposition~\ref{p.involution-family}, the involution $\Theta$ satisfies $\Theta(\mcD_i)=\mcE_i$ and $\Theta(\mcE_i)=\mcD_i$. It follows that $\Theta_0(F_0+F_i)=F_i$ and $\Theta_0(F_i)=F_0+F_i$.  Consider the Mori cone $\overline{\NE}(\mathcal{Y}_0)$, which is the numerical effective cone of curves of $\mathcal{Y}_0$. Since each $F_i$ has negative self intersection, each $F_i$ spans an extremal ray of
$\overline{\NE}(\mathcal{Y}_0)$. Then $F_0+F_i$ is an interior point of a $2$-dimensional extremal face of $\overline{\NE}(\mathcal{Y}_0)$. Since $\Theta_0$ induces an isomorphism of the Mori cone $\overline{\NE}(\mathcal{Y}_0)$, it cannot send the extremal ray of $F_i$ to the non-extremal ray of $F_0+F_i$. This contradiction shows that $\mcX_0 \simeq X(\mbA)$.
\end{proof}


\newcommand{\etalchar}[1]{$^{#1}$}
\providecommand{\bysame}{\leavevmode\hbox to3em{\hrulefill}\thinspace}

\end{document}